\documentclass[10pt]{amsart}
\usepackage{seaweed}
\usetikzlibrary{backgrounds}

\title{Combinatorics of Seaweed Meanders}
\address{Massachusetts Institute of Technology}
\email{maxlu@mit.edu}
\begin{document}

\begin{abstract}
    We show that the index of a uniformly random seaweed meander satisfies a central limit theorem and is asymptotically equidistributed with respect to any modulus. We also extend work by Archer, Geary, and Laudone on the enumeration of seaweed meanders whose compositions have restricted parts. Finally, we give a connection to billiards in triangular shapes, which is related to a result of Eu, Fu, and Hsu.
\end{abstract}
\maketitle

\section{Introduction}

Seaweed meanders were introduced by~Dergachev and Kirillov~\cite{Dergachev2000} in connection with \textit{seaweed subalgebras} of $\gl_n$. They are defined combinatorially as follows.

A composition $\lambda$ of $n$ is a sequence of positive integers with sum $n$. 
Seaweed subalgebras of $\gl_n$ are indexed by pairs of compositions. Given an ordered pair $(\lambda, \mu)$ of compositions of $n$, we construct the associated seaweed meander $\frac{\lambda}{\mu}$, which we view as a (not necessarily simple) graph on $n$ vertices, in the following way.

\begin{itemize}
    \item Draw $n$ vertices in a row.
    \item For each part $m$ of $\lambda$, in order, draw a family of nested arcs above the row, by taking the block of the next $m$ consecutive vertices, connecting the first and last vertices, second and second-to-last, and so on.
    \item Do the same for $\mu$ below the row of vertices.
\end{itemize}

More formally, suppose $\lambda = (\lambda_1, \dots, \lambda_r)$ and $\mu = (\mu_1, \dots, \mu_s)$. Take $n$ vertices in a row, labelled $1$ to $n$ from left to right. Then for each $1 \le i \le r$ and $1 \le j \le \frac{\lambda_i}{2}$, add the edge
\[(\lambda_1 + \dots + \lambda_{i - 1} + j, \lambda_1 + \dots + \lambda_{i - 1} + \lambda_i + 1 - j)\] above the row of vertices.
Likewise, for each $1 \le i' \le s$ and $1 \le j' \le \frac{\mu_{i'}}{2}$, add the edge
\[(\mu_1 + \dots + \mu_{i' - 1} + j', \mu_1 + \dots + \mu_{i' - 1} + \mu_{i'} + 1 - j')\] below the row.

\begin{ex}
    The seaweed meander $\frac{\lambda}{\mu}$ with $n = 7$, $\lambda = (4, 3)$, and $\mu = (2, 1, 1, 3)$ is given in~\cref{fig:seaweed-ex}. We also denote this seaweed meander symbolically as $\frac{4 | 3}{2 | 1 | 1 | 3}$.
    \begin{figure}[!ht]
        \begin{center}
            \resizebox{0.3\textwidth}{!}{%
                \begin{tikzpicture}[every path/.style={draw, very thick}]
                    \foreach \x in {1,...,7}
                        \draw[fill] (\x, 0) circle (0.15);
                    \draw (1,0) arc (180:0:1.5);
                    \draw (2,0) arc (180:0:0.5);
                    \draw (5,0) arc (180:0:1);
                    \draw (1,0) arc (-180:0:0.5);
                    \draw (5,0) arc (-180:0:1);
                \end{tikzpicture}
        }
        \end{center}
        \caption{The seaweed meander $\frac{4 | 3}{2 | 1 | 1 | 3}$.}
        \label{fig:seaweed-ex}
    \end{figure}
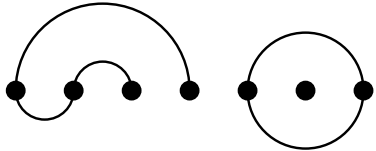
\end{ex}

Our results mainly concern the \textit{index} of a seaweed meander.

\begin{definition}\label{def:index}
    The index of a seaweed meander $\frac{\lambda}{\mu}$ is equal to the number of paths plus twice the number of cycles, where isolated vertices are counted as paths.
\end{definition}

\begin{ex}
    The index of the seaweed meander in~\cref{fig:seaweed-ex} is $4$, as there is one cycle and two paths (including the isolated vertex $6$).
\end{ex}

The definition of the index of a seaweed meander is motivated by connections to the Lie algebras. In particular, Dergachev and Kirillov~\cite{Dergachev2000} prove that the index of the associated Lie subalgebra in $\gl_n$ (which is an important statistic) is equal to the index given in~\cref{def:index}.

We also consider connections to a related class of objects known as \textit{meandric systems}.
A \textit{non-crossing arc diagram} of size $2n$ is a perfect matching on $[2n]$ such that there do not exist two pairs $(a, b)$ and $(c, d)$ with $a < c < b < d$ (i.e. no two arcs intersect). A \textit{meandric system} of size $2n$ consists of $2n$ vertices in a row, with non-crossing arc diagrams drawn above and below the vertices. We give an example in~\cref{fig:meander-ex}.

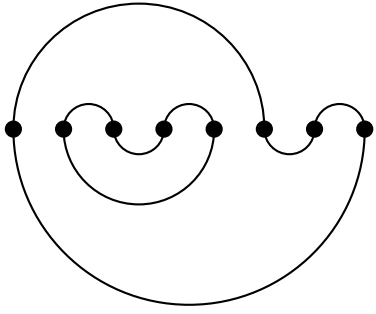
\begin{figure}[!ht]
    \begin{center}
        \resizebox{0.3\textwidth}{!}{%
            \begin{tikzpicture}[every path/.style={draw, very thick}]
                \foreach \x in {1,...,8}
                    \draw[fill] (\x, 0) circle (0.15);
                \draw (1,0) arc (180:0:2.5);
                \draw (2,0) arc (180:0:0.5);
                \draw (7,0) arc (180:0:0.5);
                \draw (4,0) arc (180:0:0.5);
                \draw (1,0) arc (-180:0:3.5);
                \draw (2,0) arc (-180:0:1.5);
                \draw (3,0) arc (-180:0:0.5);
                \draw (6,0) arc (-180:0:0.5);
            \end{tikzpicture}
    }
    \end{center}
    \caption{A meandric system of size $8$.}
    \label{fig:meander-ex}
\end{figure}

\begin{remark*}
    Note that whenever $\lambda$ or $\mu$ has an odd part, the corresponding nested arcs for that part have an isolated vertex at the center. Hence, a seaweed meander $\frac{\lambda}{\mu}$ is a meandric system if and only if both $\lambda$ and $\mu$ have only even parts.
\end{remark*}

Meandric systems have previously been studied by Arnol'd~\cite{Arnold1988ABC} and Di Francesco, Golinelli, and Guitter~\cite{francesco_meanders_1997}. One question of particular interest is the enumeration of \textit{closed meanders}, which are meandric systems with exactly one connected component. Lando and Zvonkin~\cite{lando1993plane} give a survey of various results for closed meanders and meandric systems, and a more recent perspective is given by Zvonkin in~\cite{ZVONKIN2024103817}.

\subsection{Notation}

We now define some notation which will be used to state our main results.

A composition $\lambda$ of $n$ is denoted by $\lambda \vDash n$. For a composition $\lambda \vDash n$, let $c\lambda \vDash cn$ be the composition given by scaling each term of $\lambda$ by $c$.

For two compositions $\lambda, \mu \vDash n$, the seaweed meander associated to this pair of compositions is denoted by $\frac{\lambda}{\mu}$. The index of this seaweed meander is written $\ind\left(\frac{\lambda}{\mu}\right)$, and we say that this seaweed meander has size $n$.

We take the convention that the empty seaweed meander, which has size $0$, has index $0$.

Let $\abs{M}$ denote the size of a seaweed meander $M$ and let $\calM_n$ denote the set of all seaweed meanders of size $n$. For two seaweed meanders $\frac{\lambda_1}{\mu_1}$ and $\frac{\lambda_2}{\mu_2}$, their direct sum is
\[\frac{\lambda_1}{\mu_1} \oplus \frac{\lambda_2}{\mu_2} = \frac{\lambda_1 | \lambda_2}{\mu_1 | \mu_2},\] that is, the concatenation of the two seaweed meanders.
For example, in~\cref{fig:seaweed-ex}, we have
\[\frac{4 | 3}{2 | 1 | 1 | 3} = \frac{4}{2 | 1 | 1} \oplus \frac{3}{3}.\]
Note that taking the direct sum of two seaweed meanders adds their indices, since the direct sum gives a disjoint union of the graphs.

The seaweed meander $\frac{\lambda}{\mu}$ is \textit{indecomposable} if it cannot be written as the direct sum of two smaller seaweed meanders.
Equivalently, a seaweed meander $\frac{\lambda}{\mu} \in \calM_n$ is indecomposable if $\lambda$ and $\mu$ have no common partial sums other than $0$ and $n$.

We use $\NN$ to denote the set of positive integers $\{1, 2, \dots\}$. We use $\convd$ for convergence in distribution, and $\calN(0, 1)$ for the standard normal distribution.

\subsection{Main Results}

We consider seaweed meanders from several perspectives.
We prove two probabilistic results on index distribution (\cref{thm:clt-seaweed,thm:equidistrib}), for which we make a connection to renewal-reward processes, by using a decomposition of seaweed meanders into components. We also prove a result on generating functions for seaweed meanders with restricted parts (\cref{thm:rational}), building on work by Archer, Geary, and Laudone~\cite{archer2025}. Finally, we show a connection to billiards in triangular shapes (\cref{thm:triangle-billiard}), generalizing a result by Eu, Fu, and Hsu~\cite{eu2025} from square lattices to triangular lattices.

We first discuss the distribution of the index of a uniformly random seaweed meander of size $n$. Kargin~\cite{Kargin_2020} has considered the analogous question for a uniformly random meandric system of size $2n$, conjecturing that the distribution of the number of cycles is asymptotically Gaussian.
F\'eray and Th\'evenin~\cite{feray2023} and Janson and Th\'evenin~\cite{Janson_Thévenin_2024} make partial progress towards answering this question, but the full conjecture remains open.

However, the case of seaweed meanders is easier to analyze, because a uniformly random seaweed meander decomposes (with high probability) into small indecomposable components, while a uniformly random meandric system does not exhibit such structure.

\begin{theorem}\label{thm:clt-seaweed}
    Let $X_n$ be the index of a uniformly random seaweed meander $M \in \calM_n$. Then $X_n$ satisfies a central limit theorem, that is, there exist constants $a$ and $b > 0$ such that
    \[\frac{X_n - an}{b\sqrt{n}} \convd \calN(0, 1).\]
\end{theorem}

We also obtain a result for the distribution of the index modulo any positive integer, showing convergence to the uniform distribution.

\begin{theorem}\label{thm:equidistrib}
    Fix $m \in \NN$ and $0 \le r < m$. Let $M_n$ be a uniformly random seaweed meander in $\calM_n$. Then there exists a positive constant $c = c(m) < 1$, depending only on $m$, such that
    \[\abs{\PP[\ind(M_n) \equiv r \bmod m] - \frac{1}{m}} = O(c ^ n).\]
\end{theorem}

We now consider the setting where $\lambda$ and $\mu$ have parts restricted to a fixed set $J$, as previously studied by Archer, Geary, and Laudone~\cite{archer2025}.

\begin{definition}
    Fix $J \subseteq \ZZ$. 
    Define $J ^ + = J \cap \NN$ and $J ^ - = J \setminus J ^ +$.
    Let $a_{n, k} ^ J$ be the number of seaweed meanders $\frac{\lambda}{\mu}$ of size $n$ and index $k$, where the parts of $\lambda$ and $\mu$ are restricted to the set $J_n$ defined by
    \[J_n = J ^ + \cup (n + J ^ -).\]    
    Define the generating function 
    \[A ^ J(x, y) = \sum_{n, k \ge 0} a_{n, k} ^ J x ^ n y ^ k.\]
    When $J$ is finite, say that an element of $J_n$ is \textit{large} if it is in $n + J ^ -$, and \textit{small} otherwise.
\end{definition}

Archer, Geary, and Laudone~\cite{archer2025} compute $A ^ J(x, y)$ explicitly for some choices of finite $J$. These include $J = \{1, j\}$ for $j = 2, 3, 4, 5, 7$ (for which the indecomposable seaweed meanders of size $n > j$ are necessarily acyclic) and $J = \{0, 1, 2\}$.
They suggest that the sets $J = \{-1, 1, 2\}$, $J = \{1, 6\}$, and $J = \{0, 1, 3\}$  may be interesting to consider, and we compute their corresponding generating functions in~\cref{appendix:bashing}.

\begin{remark*}
    Note that Archer, Geary, and Laudone index the generating functions by $J_n$, rather than $J$. For example, they write $A ^ {\{1, 2, n\}}(x, y)$ instead of $A ^ {\{0, 1, 2\}}(x, y)$.
\end{remark*}

Archer, Geary, and Laudone~\cite{archer2025} derive recurrences to compute rational generating functions for each aforementioned acyclic $J$. The rationality of these generating functions is in fact representative of a general result, as given by the following theorem.

\begin{theorem}\label{thm:rational}
    For any finite set $J$, the generating function $A ^ J(x, y)$ is rational.
\end{theorem}

Finally, we discuss a connection between seaweed meanders and billiards. 

A \textit{chessboard polygon} is a polygon whose edges lie on a square lattice. To define billiard trajectories in a chessboard polygon, we fire a laser starting at the midpoint of each unit segment along the boundary at an angle of $45 \dg$ to the grid, and follow the resulting reflections. The resulting paths are the \textit{billiard trajectories} of the chessboard polygon.

Likewise, for any shape lying on an equilateral triangular grid, we may also define billiard trajectories. We again begin at the midpoints of unit segments, but now fire a laser at a $60 \dg$ angle to the corresponding edge.

An example of billiard trajectories in both square and triangular grids is given in~\cref{fig:billiards-ex}.

\begin{figure}[!ht]
    \begin{center}
        \resizebox{0.6\textwidth}{!}{%
            \begin{tikzpicture}[every path/.style={draw, very thick}, subfig/.style={}]
                \node[subfig, anchor=east] (A) at (0,0){
                    \begin{tikzpicture}
                        \draw (0,0)--(0,1)--(1,1)--(1,2)--(3,2)--(3,0)--cycle;
                        \draw[gray, dotted] (1,0)--(1,1)--(3,1);
                        \draw[gray, dotted] (2,0)--(2,2);
                        \draw[red, dashed] (0.5,0)--(0,0.5)--(0.5,1)--(1.5,0)--(3,1.5)--(2.5,2)--cycle;
                        \draw[blue, dashed] (2.5,0)--(3,0.5)--(1.5,2)--(1,1.5)--cycle;
                    \end{tikzpicture}
                };
                \node[subfig, anchor=west] (B) at (1,0){
                    \begin{tikzpicture}
                        \draw (0,0)--++(60:2)--++(-60:2)--++(-120:1)--++(-1,0)--++(120:1);
                        \draw[gray, dotted] (60:1)--++(1,0)--++(-120:2);
                        \draw[gray, dotted] (0,0)--(2,0);
                        \draw[gray, dotted] (60:1)--++(-60:2);
                        \draw[red, dashed] (60:1.5)--++(0.5,0)--++(-120:2)--++(1.5,0)--cycle;
                        \draw[blue, dashed] (60:0.5)--++(1.5,0)--++(-120:1.5)--cycle;
                    \end{tikzpicture}
                };
            \end{tikzpicture}
        }
    \end{center}
    \caption{Billiard trajectories in shapes on square and triangular lattices. Both shapes have two billiard trajectories, one in red and one in blue.}
    \label{fig:billiards-ex}
\end{figure}
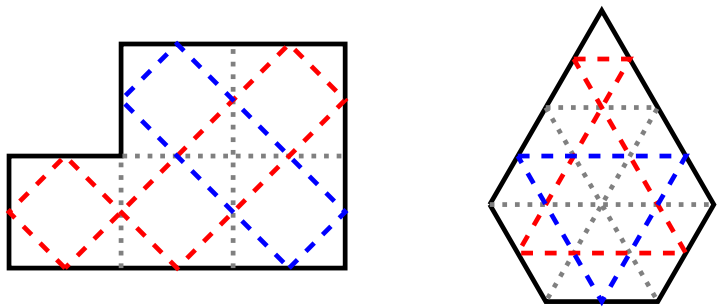

Eu, Fu, and Hsu~\cite{eu2025} prove that for a meandric system $M$, there exists a chessboard polygon constructed using the Dyck paths associated to the upper and lower arc diagrams such that the connected components of $M$ are in bijection with billiard trajectories of the chessboard polygon.

We prove that there exists a similar connection between seaweed meanders and billiard trajectories in triangular shapes, which we state more precisely as~\cref{thm:triangle-billiard} after introducing some further notation.

\subsection{Organization}

In~\cref{sec:distrib}, we prove~\cref{thm:clt-seaweed,thm:equidistrib} by first using a decomposition into indecomposable components, then connecting the index of a uniformly random seaweed meander to a renewal-reward process. We prove~\cref{thm:rational} in~\cref{sec:genfunc} using the transfer-matrix method, by exhibiting a finite automaton which can compute the index of a seaweed meander. In~\cref{appendix:bashing}, we answer a question of Archer, Geary, and Laudone~\cite{archer2025} by performing some explicit computations inspired by this method. The connection between seaweed meanders and triangular billiards is made formal and proven in~\cref{sec:billiards}. Finally, we ask a number of further questions in~\cref{sec:questions}.
\section{Index Distribution of Seaweed Meanders}\label{sec:distrib}

We first prove the central limit theorem for seaweed meanders stated in~\cref{thm:clt-seaweed}, then use the same ideas to prove the equidistribution result given in~\cref{thm:equidistrib}.

\subsection{Central Limit Theorem}
The idea for the proof of~\cref{thm:clt-seaweed} is to first show that truncating a uniformly random seaweed meander by removing the last indecomposable component changes the index by only $o(\sqrt{n})$, with high probability. We then translate the structure of the truncated seaweed meander into the language of renewal-reward processes in order to apply a central limit theorem for renewal-reward processes.

Note that sampling a uniformly random composition $\lambda \vDash n$ is equivalent to independently choosing with probability $\half$ whether each $i \in [n - 1]$ is a partial sum of $\lambda$. In particular, compositions $\lambda \vDash n$ are in bijection with subsets of $[n - 1]$, and the map is given precisely by taking the set of partial sums other than $0$ and $n$.

We now define an infinite analogue of a seaweed meander. 

\begin{definition}
    Take two infinite subsets $A, B \subseteq \NN$, where $A = \{a_1 < a_2 < \dots\}$ and $B = \{b_1 < b_2 < \dots\}$. Then the corresponding (infinite) seaweed meander $\frac{\lambda}{\mu}$ is defined analogously to the finite case, where $\lambda_i = a_i - a_{i - 1}$ and $\mu_i = b_i - b_{i - 1}$. We take the convention $a_0 = b_0 = 0$, so that the $a_i$ and $b_i$ are the sequences of partial sums for $\lambda$ and $\mu$, respectively.
\end{definition}

Now, consider the distribution $\nu$ on infinite seaweed meanders given by taking $A$, $B$ to be independently and identically distributed (iid) random subsets of $\NN$, where each $n \in \NN$ is included independently with probability $\half$.
Then the corresponding sequences $\lambda$, $\mu$ consist of independent $\Geom\left(\half\right)$ random variables, where we take the geometric distribution to have support $\NN$.

\begin{definition}
    Let $\rho$ be the distribution of the first indecomposable component of $\frac{\lambda}{\mu}$, where $\lambda$, $\mu$ consist of independent $\Geom\left(\half\right)$ random variables as above.
\end{definition}

\begin{prop}\label{prop:geom-size}
    If $I \sim \rho$, then $\abs{I} \sim \Geom\left(\frac{1}{4}\right)$.
\end{prop}

\begin{proof}
    Note that $\abs{I} = \min \{n \mid n \in A \text{ and } n \in B\}$, and the event $\{n \in A \text{ and } n \in B\}$ occurs with probability $\frac{1}{4}$, independently for each $n$. Thus, $\abs{I} \sim \Geom\left(\frac{1}{4}\right)$.
\end{proof}

We recover the distribution $\nu$ by taking the direct sum of seaweed meanders which are independent samples from $\rho$.

\begin{definition}
    Let $\{M_i\}_{i = 1} ^ \infty$ be an iid sequence of indecomposable seaweed meanders with law $\rho$, and let $\abs{M_i} = m_i$. Define
    \[Y_n = \ind\left(\bigoplus_{i = 1} ^ {T_n} M_i\right),\] where \[T_n = \max \left\{t \in \ZZ_{\ge 0} : \sum_{i = 1} ^ t m_i < n\right\},\] and we take the empty sum to be $0$.
\end{definition}

In other words, the seaweed meander $\bigoplus_{i = 1} ^ {T_n} M_i$ is given by taking a direct sum of a sequence of indecomposable components sampled from $\rho$, and truncating it at the greatest partial sum which is less than $n$. We then obtain $Y_n$ as the index of this random seaweed meander.

\begin{prop}\label{prop:truncate}
    For any $n$, $Y_n$ has the same distribution as the index of a uniformly random seaweed meander in $\calM_n$ with the last indecomposable component removed.
\end{prop}

\begin{proof}
    We may obtain both distributions by the following procedure. First, randomly sample a seaweed meander by independently choosing with probability $\half$ whether each $i \in [n - 1]$ is a partial sum of the top and bottom compositions, then truncating at the greatest common partial sum. Then the distribution is recovered by taking the index of the resulting seaweed meander.
\end{proof}

We now recall the definition of a renewal-reward process, following the exposition by Wolff~\cite{wolff}.

\begin{definition}
    A renewal process is a point process with iid waiting times $C_i$. A renewal-reward process additionally associates a \textit{reward} $R_i$ to the $i$th renewal interval $C_i$, so that the pairs $(C_i, R_i)$ are iid. (Note that $R_i$ need not be a function of $C_i$.)
    Define
    \[R(t) = \sum_{i = 1} ^ {N(t)} R_i,\]
    where
    \[N(t) = \max \left\{n \in \ZZ_{\ge 0} \colon \sum_{i = 1} ^ n C_i \le t\right\}.\]
\end{definition}

In other words, $R(t)$ is the total reward obtained up to time $t$.
In this setting, we have the following central limit theorem, which is a special case of a more general result stated in~\cite[Theorem 23]{wolff}.
\begin{theorem}\label{renewal-reward-clt}
    Consider a renewal-reward process with iid pairs $(C_i, R_i)$ of waiting times and rewards, respectively.
    Define \[a = \frac{\EE[R_1]}{\EE[C_1]}, \quad b = \sqrt{\frac{\EE[(R_1 - aC_1) ^ 2]}{\EE[C_1]}}.\] Suppose that $\EE[R_1 ^ 2] < \infty$, $\EE[C_1 ^ 2] < \infty$, and $\PP[R_1 = a C_1] < 1$. 
    Then as $t \to \infty$, 
    \[\frac{R(t) - at}{b\sqrt{t}} \convd \calN(0, 1).\]
\end{theorem}

\begin{remark*}
    Note that the third condition simply ensures that $b \ne 0$, so that the normalization is well-defined.
\end{remark*}

Now, we prove \cref{thm:clt-seaweed}, our central limit theorem, using the language of renewal-reward processes.

\begin{proof}[Proof of~\cref{thm:clt-seaweed}]
    Let $M_n$ be a uniformly random seaweed meander of size $n$, and let $\ind'(M_n)$ be the index of the resulting seaweed meander when the last indecomposable component of $M_n$ is removed.
    Note that we then have $X_n = \ind(M_n)$.
    By~\cref{prop:truncate}, $\ind'(M_n)$ has the same distribution as $Y_n$.

    We first show
    \[\frac{\ind(M_n) - \ind'(M_n)}{\sqrt{n}} \convp 0.\]
    Note that $\ind(M_n) - \ind'(M_n)$ is exactly the index of the last indecomposable component $I$ of $M_n$.
    By the same argument as in~\cref{prop:geom-size}, we have \[\abs{I} \sim \min\left(\Geom\left(\frac{1}{4}\right), n\right).\]
    Hence, for any $\eps > 0$, we have
    \[\PP[\abs{I} > \eps \sqrt{n}] \le \left(\frac{3}{4}\right) ^ {\floor{\eps \sqrt{n}}} \to 0.\]
    For any indecomposable component $I$, we have $\ind(I) \le \abs{I}$.
    Thus,
    \[\PP\left[\frac{\ind(M_n) - \ind'(M_n)}{\sqrt{n}} > \eps\right] = \PP[\ind(I) > \eps \sqrt{n}] \le \PP[\abs{I} > \eps \sqrt{n}] \to 0,\] so
    \[\frac{\ind(M_n) - \ind'(M_n)}{\sqrt{n}} \convp 0.\]
    Now, note that $Y_n$ is exactly the reward $R(n - 1)$ for a reward-renewal process which takes $I_i \sim \rho$, and has waiting times $C_i = \abs{I_i}$ and rewards $R_i = \ind(I_i)$.

    We have $0 < R_i \le C_i$, and by~\cref{prop:geom-size}, $C_i \sim \Geom\left(\frac{1}{4}\right)$. Thus, $\EE[C_i ^ 2] < \infty$ and $\EE[R_i ^ 2] < \infty$.
    Finally, we check that $\frac{R_i}{C_i}$ is non-constant. This clearly holds, as the seaweed meanders $\frac{2}{2}$ and $\frac{2}{1 | 1}$ are both of size $2$ but have index $2$ and $1$, respectively.

    Thus, by the central limit theorem for renewal-reward processes stated above in~\cref{renewal-reward-clt}, we have
    \[\frac{Y_n - a(n - 1)}{b\sqrt{n - 1}} = \frac{R(n - 1) - a(n - 1)}{b\sqrt{n - 1}} \convd \calN(0, 1),\] where
    \[a = \frac{\EE[R_1]}{\EE[C_1]}, \quad b = \sqrt{\frac{\EE[(R_1 - aC_1) ^ 2]}{\EE[C_1]}}.\]
    By~\cref{prop:truncate},
    \[\frac{\ind'(M_n) - a(n - 1)}{b\sqrt{n - 1}} \convd \calN(0, 1).\]
    By Slutsky's theorem, we may replace $b\sqrt{n - 1}$ in the denominator with $b\sqrt{n}$, and
    \[\frac{\ind(M_n) - an}{b\sqrt{n}} - \frac{\ind'(M_n) - a(n - 1)}{b\sqrt{n}} = \frac{\ind(M_n) - \ind'(M_n)}{b\sqrt{n}} - \frac{a}{b\sqrt{n}} \convp 0,\] so
    \[\frac{X_n - an}{b\sqrt{n}} = \frac{\ind(M_n) - an}{b\sqrt{n}} \convd \calN(0, 1),\] as desired.
\end{proof}

\begin{remark*}
    Numerical computations suggest $a \approx 0.42$ and $b \approx 0.49$, but it seems difficult to compute closed-form expressions for $a$ and $b$. These estimates were obtained by enumerating all indecomposable components of bounded size, then applying the formulas given by the renewal-reward central limit theorem.
\end{remark*}

The same argument also gives a central limit theorem for various other statistics which are additive over indecomposable components.
\begin{cor}\label{cor:more-clt}
    Each of the following statistics on a uniformly random seaweed meander of size $n$ also satisfies a central limit theorem as $n \to \infty$:
    \begin{itemize}
        \item number of cycles,
        \item number of paths,
        \item number of occurrences of a fixed connected component (up to translation),
        \item number of indecomposable components of a fixed size.
    \end{itemize}
\end{cor}

\begin{proof}
    The same argument as in the proof of~\cref{thm:clt-seaweed} works, where we now take the reward to be the desired statistic. It is straightforward to check that the ratio $\frac{R_i}{C_i}$ is non-constant in each case.
\end{proof}

\subsection{Equidistribution}

We now show that the index of a uniformly random seaweed meander is approximately equidistributed for any modulus. The proof proceeds in two steps. We first show in~\cref{lem:size2-comp} that with high probability, a uniformly random seaweed meander of size $n$ has $\Omega(n)$ indecomposable components of size $2$. Then we vary the component among each of the three indecomposable components of size $2$, two of which have index $1$ and one of which has index $2$, at each of these $\Omega(n)$ positions to derive the equidistribution result.

\begin{definition}
    For a seaweed meander $M$, let $f_2(M)$ be the number of indecomposable components of $M$ of size $2$.
\end{definition}

We say two seaweed meanders of size $n$ are equivalent if they have indecomposable components of the same sizes, in the same order, so that all components of size not equal to $2$ agree.

\begin{lemma}
    For a seaweed meander $M$, the number of seaweed meanders in its equivalence class is exactly $3 ^ {f_2(M)}$.
\end{lemma}

\begin{proof}
    There are exactly three indecomposable components of size $2$, namely $\frac{2}{2}$, $\frac{2}{1 | 1}$, and $\frac{1 | 1}{2}$. At each of the $f_2(M)$ components of size $2$, we may choose any of these three components, for a total of $3 ^ {f_2(M)}$ seaweed meanders equivalent to $M$.
\end{proof}

\begin{ex}
    The seaweed meanders
    \[\frac{2 | 3}{1 | 1 | 2 | 1}, \frac{1 | 1 | 3}{2 | 2 | 1}, \frac{2 | 3}{2 | 2 | 1}\] form a complete equivalence class in $\calM_5$.
\end{ex}

Let the equivalence classes in $\calM_n$ be $\calE = \{E_1, \dots, E_\ell\}$, and define $f_2(E_i) = f_2(M)$ for each $M \in E_i$. This is well-defined, since all seaweed meanders in $E_i$ have the same number of components of size $2$.

\begin{definition}
    Let $M_{E_i}$ be the seaweed meander obtained when each indecomposable component of size $2$ of a meander in $E_i$ is replaced by the component $\frac{2}{2}$.
\end{definition}
This is well-defined because all seaweed meanders in $E_i$ agree on components which do not have size $2$, and hence produce the same seaweed meander under this operation.

We first show that with high probability, a uniformly random seaweed meander of size $n$ has $\Omega(n)$ indecomposable components of size $2$.

\begin{lemma}\label{lem:size2-comp}
    There exist constants $\eps, \delta > 0$ such that for all sufficiently large $n$,
    \[\PP[f_2(M_n) < \eps n] \le e ^ {-\delta n},\] where $M_n$ is a uniformly random seaweed meander of size $n$.
\end{lemma}

\begin{proof}
    Let $\indic_i$ be the indicator variable for an indecomposable component of size $2$ consisting of the vertices $i$ and $i + 1$. Then
    \[f_2 = \indic_1 + \dots + \indic_{n - 1}.\]
    
    As before, we may choose a uniformly random seaweed meander by randomly selecting the sequences of partial sums for $\lambda$ and $\mu$. Then the indicator $\indic_i$ depends only on whether $i - 1$, $i$, and $i + 1$ are partial sums of $\lambda$ and $\mu$.

    In particular, $\indic_i = 1$ if and only if $i - 1$ and $i + 1$ are partial sums of both $\lambda$ and $\mu$, while $i$ is a partial sum for at most one of $\lambda$ and $\mu$. Hence, for $2 \le i \le n - 2$, each indicator $\indic_i$ is Bernoulli with parameter $\frac{1}{4} \cdot \frac{3}{4} \cdot \frac{1}{4} = \frac{3}{64}$.
    Furthermore, the above argument shows that the random variables $\{\indic_{3i}\}$ are independent. Thus,
    \[f_2(M_n) \ge \indic_3(M_n) + \dots + \indic_{3\floor{\frac{n - 2}{3}}}(M_n).\]
    In particular, $f_2(M_n)$ is bounded below by the sum of $\Omega(n)$ iid Bernoulli random variables, so applying the Chernoff bound gives the desired tail bound.
\end{proof}

Using this lemma, we now prove~\cref{thm:equidistrib}.

\begin{proof}[Proof of~\cref{thm:equidistrib}]
    Fix an equivalence class $E_i$, and let $M$ be a uniformly random seaweed meander chosen from $E_i$.

    For any indecomposable component of size $2$, there is a $\frac{2}{3}$ probability it appears as $\frac{2}{1 | 1}$ or $\frac{1 | 1}{2}$ in $M$, and a $\frac{1}{3}$ probability it appears as $\frac{2}{2}$. In the former case, the component contributes index $1$, while in the latter case, it contributes index $2$.

    Hence, we may sample the index of a uniformly random seaweed meander in $E_i$ in the following way. Start with $\ind(M_{E_i})$, and take $f_2(E_i)$ steps, where the index remains unchanged with probability $\frac{1}{3}$ and decreases by $1$ with probability $\frac{2}{3}$.

    Reducing mod $m$, we have an irreducible and aperiodic Markov chain on the $m$ residue classes, whose unique stationary distribution is the uniform distribution. It is well-known that irreducible, aperiodic Markov chains converge exponentially quickly to their stationary distributions, that is, if $X_s$ is the index after $s$ steps, we have
    \[\abs{\PP[X_s \equiv r \bmod m] - \frac{1}{m}} = e ^ {-\Omega(s)}.\]
    In particular, if $f_2(E_i) \ge \eps n$, we have
    \[\abs{\PP[\ind(M) \equiv r \bmod m] - \frac{1}{m}} = e ^ {-\Omega(n)}.\]
    By~\cref{lem:size2-comp}, we have the bound
    \[\PP[f_2(M_n) < \eps n] = e ^ {-\Omega(n)}.\]
    Thus, we have
    \begin{align*}
        \abs{\PP[\ind(M_n) \equiv r \bmod m] - \frac{1}{m}} &= \abs{\sum_{i = 1} ^ \ell \left(\PP[\ind(M_n) \equiv r \bmod m \mid M_n \in E_i] - \frac{1}{m}\right) \cdot \PP[M_n \in E_i]} \\
        &\le \sum_{i = 1} ^ \ell \abs{\left(\PP[\ind(M_n) \equiv r \bmod m \mid M_n \in E_i] - \frac{1}{m}\right) \cdot \PP[M_n \in E_i]} \\
        &= \sum_{i \colon f_2(E_i) < \eps n} \abs{\left(\PP[\ind(M_n) \equiv r \bmod m \mid M_n \in E_i] - \frac{1}{m}\right) \cdot \PP[M_n \in E_i]} \\
        &\phantom{=} + \sum_{i \colon f_2(E_i) \ge \eps n} \abs{\left(\PP[\ind(M_n) \equiv r \bmod m \mid M_n \in E_i] - \frac{1}{m}\right) \cdot \PP[M_n \in E_i]} \\
        &\le \sum_{i \colon f_2(E_i) < \eps n} \PP[M_n \in E_i] + \sum_{i \colon f_2(E_i) \ge \eps n} e ^ {-\Omega(n)}\PP[M_n \in E_i] \\
        &\le \PP[f_2(M_n) < \eps n] + e ^ {-\Omega(n)}\PP[f_2(M_n) \ge \eps n] \\
        &= O(c ^ n)
    \end{align*} for some constant $c < 1$.
\end{proof}
\section{Generating Functions}\label{sec:genfunc}

We consider seaweed meanders associated to compositions with restricted parts, as studied by Archer, Geary, and Laudone~\cite{archer2025}. 

\subsection{Transfer-Matrix Method} We first recall some definitions and notation from~\cite{stanley}.

\begin{definition}[{\cite[Section 4.7]{stanley}}]\label{def:automaton}
    Let $R$ be a commutative ring, and let $G$ be a finite directed graph (not necessarily simple) with a weight function $w \colon E(G) \to R$. For a walk $\Gamma = e_1\dots e_n$ on $G$, define $\abs{\Gamma} = n$ and $w(\Gamma) = w(e_1) \dots w(e_n)$.
    For $i, j \in V(G)$, define
    \[a_{ij}(n) = \sum_\Gamma w(\Gamma),\] where the sum is over all walks $\Gamma$ from $i$ to $j$ of length $n$, and
    \[A_{ij}(x) = \sum_{n \ge 0} a_{ij}(n) x ^ n \in R[[x]].\]
\end{definition}

We will use the following result.

\begin{theorem}[{\cite[Theorem 4.7.2]{stanley}}]\label{thm:stanley-gen-func}
    Let $R$ be a commutative ring, and $G$ be a finite directed graph with edge weights in $R$. Let $A = \{a_{ij}(1)\}$ be the (weighted) adjacency matrix of $G$. Then
    \[A_{ij}(x) = \frac{C_{ji}(I - xA)}{\det(I - xA)},\] where $C_{ji}(M)$ is the cofactor of $M$ obtained by removing the $j$th row and $i$th column of $M$.
\end{theorem}

\subsection{Proof Outline}

We now begin the setup for the proof of~\cref{thm:rational}. For the remainder of this section, $J$ will always be a finite set.
We split into cases depending on whether neither $\lambda$ nor $\mu$ has a large part, exactly one of $\lambda$ and $\mu$ has a large part, or both $\lambda$ and $\mu$ have large parts.

Throughout the proof, take $N_0 = N_0(J)$ to be sufficiently large, so that there is no ambiguity between large and small parts for all $n \ge N_0$. We now restrict our attention to $n \ge N_0$.

\begin{definition}
    Let $b_{n, k} ^ J$ be the number of seaweed meanders $\frac{\lambda}{\mu}$ of size $n$ and index $k$ such that $\lambda$ and $\mu$ have all parts in $J_n$ and each has at most one large part. Then define
    \[B ^ J(x, y) = \sum_{n \ge N_0} \sum_{k \ge 0} b_{n, k} ^ J x ^ n y ^ k.\]
    Furthermore, for $i \in \{0, 1, 2\}$, let
    \[b_{n, k, i} ^ J\] be the number of seaweed meanders $\frac{\lambda}{\mu}$ of size $n$ and index $k$ such that exactly $i$ of $\{\lambda, \mu\}$ have a large part, and define the corresponding generating functions $B_i ^ J(x, y)$ in the natural way.
\end{definition}

By definition, we have
\[b_{n, k} ^ J = b_{n, k, 0} ^ J + b_{n, k, 1} ^ J + b_{n, k, 2} ^ J.\]
We first show that it is enough to consider $B ^ J$ to prove rationality.

\begin{prop}\label{prop:ab-equiv}
    The generating function $A ^ J$ is rational if and only if $B ^ J$ is rational.
\end{prop}

\begin{proof}
    Note that for all sufficiently large $n$, any composition of $n$ has at most one part whose size is in $n + J ^ -$. Hence, for all large enough $n$, $a_{n, k} ^ J = b_{n, k} ^ J$, so $A ^ J$ and $B ^ J$ differ by a polynomial.
\end{proof}

Thus, it suffices to consider the seaweed meanders enumerated by $b_{n, k} ^ J$. We prove the rationality of $B ^ J$ by proving that each of the generating functions $B_i ^ J$ is rational.

\begin{lemma}\label{lem:zero-sum-seq}
    For sufficiently large $n$, the compositions $\lambda \vDash n$ such that $\lambda$ has exactly one large part are in bijection with sequences of elements of $J$ with sum $0$ containing exactly one non-positive term.
    In particular, the number of such compositions is bounded by a constant depending only on $J$.
\end{lemma}

\begin{proof}
    For any such composition $\lambda$, replacing the part $n + j$ with $j \in J ^ -$ gives a sequence with sum $0$ and exactly one non-positive part.
    
    For the second half, the number of such zero-sum sequences is bounded, since if the non-positive term is $j \in J ^ -$, the sequence has length at most $\abs{j} + 1$.
\end{proof}

Using the above lemma, we can now handle the case where $J$ has non-positive elements. In particular, since $J$ is finite, we can subdivide the cases where $\lambda$ or $\mu$ has a large part based on which element of $n + J ^ -$ appears.

The following definition will be useful for the cases where at least one of $\lambda$ or $\mu$ has a large part.

\begin{definition}
    Let $\calS_J$ be the set of zero-sum sequences with elements in $J$ that have exactly one non-positive element.
    For $s \in \calS_J$, let $C_n(s)$ be the associated composition of $n$, by adding $n$ to the non-positive term.
\end{definition}

In the following sections, we will use the language of (deterministic) finite automata. In particular, the notions of \textit{states}, \textit{transitions}, and an \textit{input string} will be important. For a standard reference, see~\cite[Chapter 1]{Sipser_2013}.

To prove that $B_0 ^ J$ and $B_1 ^ J$ are rational, we will first describe an algorithm which computes the index of the appropriate seaweed meanders. Then by showing that these algorithms require only finitely many states, we can translate the algorithms into finite automata. Finally, by assigning edge weights to the underlying graphs, we will be able to apply~\cref{thm:stanley-gen-func}. For $B_2 ^ J$, a combinatorial argument will prove rationality directly.

\subsection{Rationality of $B_0 ^ J$}\label{subsec:b0}
We first consider the case where neither $\lambda$ nor $\mu$ has a large part.

\begin{lemma}\label{lem:ab-positive}
    For any finite $J \subseteq \ZZ$, $B_0 ^ J$ is rational if and only if $A ^ {J ^ +}$ is rational.
\end{lemma}

\begin{proof}
    As in the proof of~\cref{prop:ab-equiv}, $B_0 ^ J$ and $A ^ {J ^ +}$ differ by a polynomial, so the result follows.
\end{proof}

Thus, it suffices to consider the case in which $J \subseteq \NN$, which we now state.

\begin{prop}\label{prop:rational-positive}
    For any finite $J \subseteq \NN$, the generating function $A ^ J(x, y)$ is rational.
\end{prop}

We now describe an algorithm which computes the index of a seaweed meander. The idea is to encode a seaweed meander of size $n$ as a string of length $n$, then construct a finite automaton which reads this string as input.
Using this finite automaton, we can then use~\cref{thm:stanley-gen-func} to deduce that the generating function $A ^ J$ is rational.

\begin{definition}
    For a composition $\lambda = (\lambda_1, \dots, \lambda_\ell) \vDash n$, define $f(\lambda)$ to be the sequence of length $n$ where $\lambda_1$ is repeated $\lambda_1$ times, followed by $\lambda_2$ repeated $\lambda_2$ times, and so on. For a seaweed meander $\frac{\lambda}{\mu}$, define the sequence of formal symbols \[f\left(\frac{\lambda}{\mu}\right) = \left(\frac{f(\lambda)_i}{f(\mu)_i}\right)_{i = 1} ^ n.\]
\end{definition}
Note that this encoding is injective, so it encodes the entire data of the seaweed meander. We will use this encoding to compute the index.

We take as input the string $f\left(\frac{\lambda}{\mu}\right)$. We start with an empty initial state, with no vertices or edges. Upon reading each input character, we add a new vertex at the right of the current state, initially with no edges adjacent to it.

If that input character completes a part in either $\lambda$ or $\mu$, add the corresponding edges to the current state. Then we may remove any vertices for which the parts in both $\lambda$ and $\mu$ have been completed, updating the index as required.

Note that when the algorithm terminates, we are again in the empty initial state, since all vertices will have been accounted for in both $\lambda$ and $\mu$.

In~\cref{ex:exec-1,fig:execution}, we show how such a computation may proceed.

\begin{ex}\label{ex:exec-1}
    Take $n = 7$, $\lambda = (5, 2)$ and $\mu = (2, 4, 1)$. 
    Then
    \[f\left(\frac{\lambda}{\mu}\right) = \left(\frac{5}{2}, \frac{5}{2}, \frac{5}{4}, \frac{5}{4}, \frac{5}{4}, \frac{2}{4}, \frac{2}{1}\right).\]
    The computation is then given by the process shown in~\cref{fig:execution}.
    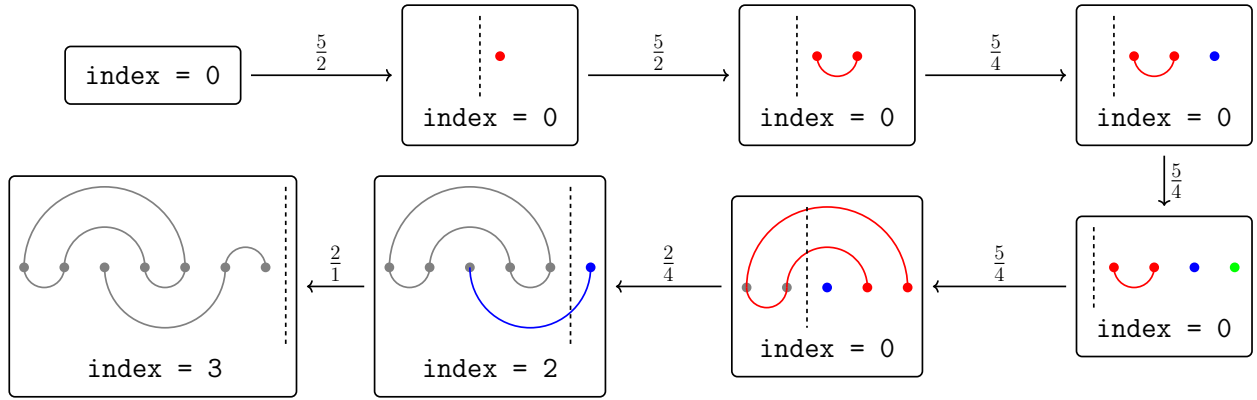
\begin{figure}[!ht]
        \begin{center}
            \resizebox{\textwidth}{!}{%
                \begin{tikzpicture}[every node/.style={font=\ttfamily\Huge},every path/.style={draw, very thick},subfig/.style={draw, rounded corners, inner sep=7pt}]
                    \tikzmath{\r = 0.1;\d=0.25;\gap=4;}
                    \node[subfig] (A) at (0,0){
                        \begin{tikzpicture}
                            \node at (0,0) {index = 0};
                        \end{tikzpicture}
                    };
                    \node[subfig, anchor=west] (B) at ($(A.east)+(\gap,0)$){
                        \begin{tikzpicture}
                            \fill[red] (1,0) circle (\r);
                            \node[anchor=center] at (0.75,-1.5) {index = 0};
                            \draw[dashed] (0.5, 1)--(0.5, -1);
                        \end{tikzpicture}
                    };
                    \node[subfig, anchor=west] (C) at ($(B.east)+(\gap,0)$){
                        \begin{tikzpicture}
                            \fill[red] (1,0) circle (\r);
                            \fill[red] (2,0) circle (\r);
                            \node[anchor=center] at (1.25,-1.5) {index = 0};
                            \draw[red] (1,0) arc (-180:0:0.5);
                            \draw[dashed] (0.5, 1)--(0.5, -1);
                        \end{tikzpicture}
                    };
                    \node[subfig, anchor=west] (D) at ($(C.east)+(\gap,0)$){
                        \begin{tikzpicture}
                            \fill[red] (1,0) circle (\r);
                            \fill[red] (2,0) circle (\r);
                            \fill[blue] (3,0) circle (\r);
                            \node[anchor=center] at (1.75,-1.5) {index = 0};
                            \draw[red] (1,0) arc (-180:0:0.5);
                            \draw[dashed] (0.5, 1)--(0.5, -1);
                        \end{tikzpicture}
                    };
                    \node[subfig, anchor=north] (E) at ($(D.south)-(0,1.75)$){
                        \begin{tikzpicture}
                            \fill[red] (1,0) circle (\r);
                            \fill[red] (2,0) circle (\r);
                            \fill[blue] (3,0) circle (\r);
                            \fill[green] (4,0) circle (\r);
                            \node[anchor=center] at (2.25,-1.5) {index = 0};
                            \draw[red] (1,0) arc (-180:0:0.5);
                            \draw[dashed] (0.5, 1)--(0.5, -1);
                        \end{tikzpicture}
                    };
                    \node[subfig] (F) at (C |- E){
                        \begin{tikzpicture}
                            \fill[gray] (1,0) circle (\r);
                            \fill[gray] (2,0) circle (\r);
                            \fill[blue] (3,0) circle (\r);
                            \fill[red] (4,0) circle (\r);
                            \fill[red] (5,0) circle (\r);
                            \node[anchor=center] at (3,-1.5) {index = 0};
                            \draw[red] (1,0) arc (-180:0:0.5);
                            \draw[red] (1,0) arc (180:0:2);
                            \draw[red] (2,0) arc (180:0:1);
                            \draw[dashed] (2.5, 2)--(2.5, -1);
                        \end{tikzpicture}
                    };
                    \node[subfig] (G) at (B |- E){
                        \begin{tikzpicture}
                            \foreach \x in {1,...,5}
                                \fill[gray] (\x,0) circle (\r);
                            \fill[blue] (6,0) circle (\r);
                            \node[anchor=center] at (3.5,-2.5) {index = 2};
                            \draw[gray] (1,0) arc (-180:0:0.5);
                            \draw[gray] (4,0) arc (-180:0:0.5);
                            \draw[gray] (1,0) arc (180:0:2);
                            \draw[gray] (2,0) arc (180:0:1);
                            \draw[blue] (3,0) arc (-180:0:1.5);
                            \draw[dashed] (5.5, 2)--(5.5, -2);
                        \end{tikzpicture}
                    };
                    \node[subfig] (H) at (A |- E){
                        \begin{tikzpicture}
                            \foreach \x in {1,...,7}
                                \fill[gray] (\x,0) circle (\r);
                            \node[anchor=center] at (4.25,-2.5) {index = 3};
                            \draw[gray] (1,0) arc (-180:0:0.5);
                            \draw[gray] (4,0) arc (-180:0:0.5);
                            \draw[gray] (1,0) arc (180:0:2);
                            \draw[gray] (2,0) arc (180:0:1);
                            \draw[gray] (6,0) arc (180:0:0.5);
                            \draw[gray] (3,0) arc (-180:0:1.5);
                            \draw[dashed] (7.5, 2)--(7.5, -2);
                        \end{tikzpicture}
                    };
                    \draw[->] ($(A.east)+(\d,0)$)--node[above, font=\Huge]{$\frac{5}{2}$}($(B.west)-(\d,0)$);
                    \draw[->] ($(B.east)+(\d,0)$)--node[above, font=\Huge]{$\frac{5}{2}$}($(C.west)-(\d,0)$);
                    \draw[->] ($(C.east)+(\d,0)$)--node[above, font=\Huge]{$\frac{5}{4}$}($(D.west)-(\d,0)$);
                    \draw[->] ($(D.south)-(0,\d)$)--node[right, font=\Huge]{$\frac{5}{4}$}($(E.north)+(0,\d)$);
                    \draw[->] ($(E.west)-(\d,0)$)--node[above, font=\Huge]{$\frac{5}{4}$}($(F.east)+(\d,0)$);
                    \draw[->] ($(F.west)-(\d,0)$)--node[above, font=\Huge]{$\frac{2}{4}$}($(G.east)+(\d,0)$);
                    \draw[->] ($(G.west)-(\d,0)$)--node[above, font=\Huge]{$\frac{2}{1}$}($(H.east)+(\d,0)$);
                \end{tikzpicture}
            }
        \end{center}
        \caption{Index computation for $\lambda = (5, 2)$, $\mu = (2, 4, 1)$. Each color corresponds to a connected component of the current state, and the vertices to the right of the dashed line are those which are stored in the current state.}
        \label{fig:execution}
    \end{figure}
\end{ex}

We note that there is one subtlety in the step where we update the index. A cycle occurs in the seaweed meander if and only if a connected component is completed by an arc which connects two vertices already in the same connected component. Otherwise, the connected component is a path. This is something which we can account for using the algorithm described above, since we track the connected components of all vertices which are in parts that are not yet complete.

In the example, this occurs when the algorithm completes the part of size $4$ in $\mu$ and adds an arc between vertices $4$ and $5$, which were already connected.

We now prove that $A ^ J$ is rational whenever $J \subseteq \NN$.

\begin{proof}[Proof of~\cref{prop:rational-positive}]
    Fix $J$. The case where $J = \emptyset$ is clear, so assume $J$ is non-empty.
    
    Let $C = \max \{j \in J\}$. 
    Note that the state tracks only vertices such that some, but not all vertices from their parts in $\lambda$ and $\mu$ have been seen in the input. Since all parts are bounded by $C$, the state tracks at most $C$ vertices at once.

    Each state corresponds to some set partition on the tracked vertices, along with the data (for each of $\lambda$ and $\mu$) of the size of the current part and the number of vertices in that part which have been processed. The latter two quantities are bounded by $C$, so there are finitely many possible states in total.

    Thus, the algorithm has only finitely many possible states, so there exists a finite automaton where edges are labelled by input characters such that the transitions agree with the above algorithm.

    Now, for a transition which increases the index by $i$, assign a weight of $y ^ i$ to that edge. The walks of length $n$ which start and end at the empty state then correspond exactly to seaweed meanders of size $n$ with parts in $J$. Furthermore, if the seaweed meander has index $k$, the walk has weight $y ^ k$.

    Thus, using the notation from~\cref{thm:stanley-gen-func}, we have
    \[A ^ J(x, y) = \sum_M x ^ {\abs{M}} y ^ {\ind(M)} = \sum_{\Gamma} w(\Gamma) x ^ {\abs{\Gamma}},\] where the first sum is over all seaweed meanders $M$ with parts in $J$, and the second sum is over all walks $\Gamma$ from the empty state to itself.

    By~\cref{thm:stanley-gen-func} with $R = \QQ[y]$, the generating function for such walks is rational, so $A ^ J(x, y)$ is rational.
\end{proof}

\subsection{Rationality of $B_1 ^ J$}

We now consider the case where exactly one of $\lambda$ and $\mu$ has a large part. Fix a finite subset $J \subseteq \ZZ$.

\begin{prop}\label{prop:b1-rational}
    For any finite $J \subseteq \ZZ$, $B_1 ^ J$ is rational.
\end{prop}

Fix some $s \in \calS_J$, and suppose $\mu = C_n(s)$. (The other case is symmetric.)

We show how to compute the index of such a seaweed meander, using a method similar to the case where $J \subseteq \NN$. We again encode a seaweed meander into a string, which now will have length $n + 1$ instead of $n$.

Since we fix $s$, and hence $\mu$, the input only needs to encode the data of $\lambda$. Define $f(\lambda)$ as before, and let $g(\lambda)$ be the sequence obtained by reading the elements of $f(\lambda)$ in the order
\[f(\lambda)_1, f(\lambda)_n, f(\lambda)_2, f(\lambda)_{n - 1}, \dots, f(\lambda)_{\floor{\frac{n}{2}} + 1}, \#,\] where we append the end-of-string character $\#$.

\begin{ex}
    Again taking $\lambda = (5, 2)$, we have
    \[g(\lambda) = (5, 2, 5, 2, 5, 5, 5, \#).\] 
\end{ex}

We now use $g(\lambda)$ as the input to our algorithm.
As in~\cref{subsec:b0}, when reading the input $g(\lambda)$, we keep track of the current connected components as part of the state. However, since $g(\lambda)$ reorders the vertices in the input, we now have vertices from both the left and right ends of the seaweed meander.

Suppose the large part of $\mu$ contains the vertices in the interval $[a, n - b]$.
Since $a$ and $b$ are constants (independent of $n$), we can tell whether or not a vertex is in the large part of $\mu$ simply by counting the first $2\max(a, b)$ input characters.

We also update the arcs differently than before. For any part of $\lambda$ or small part of $\mu$, we add the arcs corresponding to that part whenever all vertices in that part have been read as input, as in the earlier algorithm. However, for each vertex $v \in [a, n - b]$, we add the edge $(v, n + a - b - v)$ once the input characters for both those vertices have been processed.

The end-of-string character $\#$ is required to combine the left and right halves when the input is complete, since the edges crossing $\frac{n}{2}$ cannot be added until the entire input is read, as demonstrated by~\cref{ex:alg-large-parts}.

We now give an example illustrating the above algorithm.
\begin{ex}\label{ex:alg-large-parts}
    Take $n = 6$, $\lambda = (1, 3, 2)$, and $s = (-2, 2)$, so that $\mu = C_6(s) = (4, 2)$. Then $f(\lambda) = (1, 3, 3, 3, 2, 2)$, so
    \[g(\lambda) = (1, 2, 3, 2, 3, 3, \#).\]
    The computation then proceeds as shown in~\cref{fig:alg-large-parts} Note that when the final character $\#$ is read, we add the edges $(2, 4)$ above and $(2, 3)$ below, since we now know that the three $3$s in the input come from a single part of $\lambda$. However, until we read the final input character, the vertices may still come from different parts of size $3$, so we cannot add the arcs in the second-to-last step.
    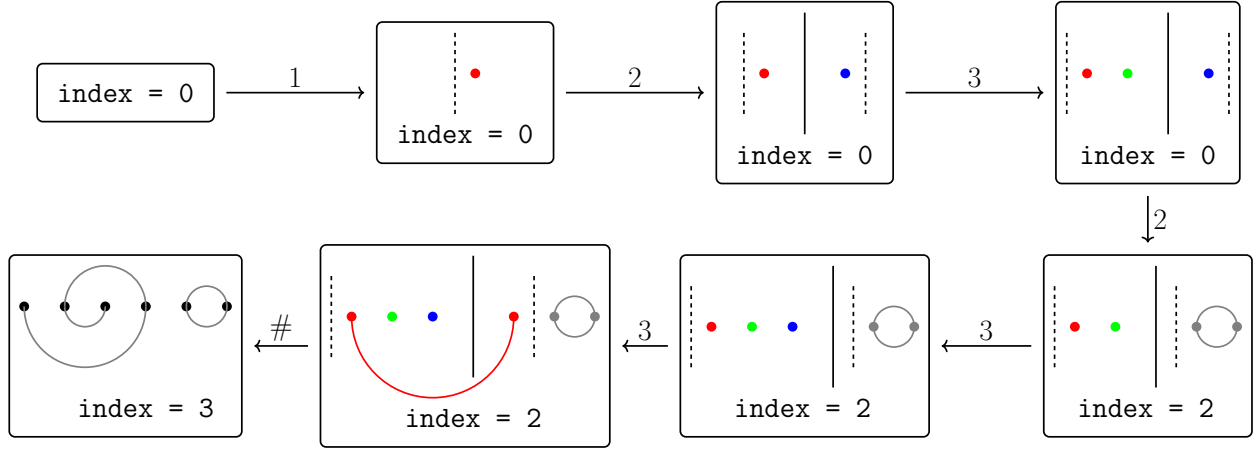
\begin{figure}[!ht]
        \begin{center}
            \resizebox{\textwidth}{!}{%
                \begin{tikzpicture}[every node/.style={font=\ttfamily\Huge},every path/.style={draw, very thick},subfig/.style={draw, rounded corners, inner sep=7pt}]
                    \tikzmath{\r = 0.1;\d=0.3;\gap=4;}
                    \node[subfig] (A) at (0,0){
                        \begin{tikzpicture}
                            \node at (0,0) {index = 0};
                        \end{tikzpicture}
                    };
                    \node[subfig, anchor=west] (B) at ($(A.east)+(\gap,0)$){
                        \begin{tikzpicture}
                            \fill[red] (1,0) circle (\r);
                            \node[anchor=center] at (0.75,-1.5) {index = 0};
                            \draw[dashed] (0.5, 1)--(0.5, -1);
                        \end{tikzpicture}
                    };
                    \node[subfig, anchor=west] (C) at ($(B.east)+(\gap,0)$){
                        \begin{tikzpicture}
                            \fill[red] (1,0) circle (\r);
                            \fill[blue] (3,0) circle (\r);
                            \node[anchor=center] at (2,-2) {index = 0};
                            \draw[dashed] (0.5, 1)--(0.5, -1);
                            \draw[dashed] (3.5, 1)--(3.5,-1);
                            \draw (2,1.5)--(2,-1.5);
                        \end{tikzpicture}
                    };
                    \node[subfig, anchor=west] (D) at ($(C.east)+(\gap,0)$){
                        \begin{tikzpicture}
                            \fill[red] (1,0) circle (\r);
                            \fill[green] (2, 0) circle (\r);
                            \fill[blue] (4,0) circle (\r);
                            \node[anchor=center] at (2.5,-2) {index = 0};
                            \draw[dashed] (0.5, 1)--(0.5, -1);
                            \draw[dashed] (4.5, 1)--(4.5,-1);
                            \draw (3,1.5)--(3,-1.5);
                        \end{tikzpicture}
                    };
                    \node[subfig, anchor=north] (E) at ($(D.south)-(0,1.75)$){
                        \begin{tikzpicture}
                            \fill[red] (1,0) circle (\r);
                            \fill[green] (2, 0) circle (\r);
                            \fill[gray] (4,0) circle (\r);
                            \fill[gray] (5,0) circle (\r);
                            \node[anchor=center] at (2.75,-2) {index = 2};
                            \draw[dashed] (0.5, 1)--(0.5, -1);
                            \draw[dashed] (3.5, 1)--(3.5,-1);
                            \draw (3,1.5)--(3,-1.5);
                            \draw[gray] (4.5,0) circle (0.5);
                        \end{tikzpicture}
                    };
                    \node[subfig] (F) at (C |- E){
                        \begin{tikzpicture}
                            \fill[red] (1,0) circle (\r);
                            \fill[green] (2, 0) circle (\r);
                            \fill[blue] (3,0) circle (\r);
                            \fill[gray] (5,0) circle (\r);
                            \fill[gray] (6,0) circle (\r);
                            \node[anchor=center] at (3.25,-2) {index = 2};
                            \draw[dashed] (0.5, 1)--(0.5, -1);
                            \draw[dashed] (4.5, 1)--(4.5,-1);
                            \draw (4,1.5)--(4,-1.5);
                            \draw[gray] (5.5,0) circle (0.5);
                        \end{tikzpicture}
                    };
                    \node[subfig] (G) at (B |- E){
                        \begin{tikzpicture}
                            \fill[red] (1,0) circle (\r);
                            \fill[green] (2, 0) circle (\r);
                            \fill[blue] (3,0) circle (\r);
                            \fill[red] (5,0) circle (\r);
                            \fill[gray] (6,0) circle (\r);
                            \fill[gray] (7,0) circle (\r);
                            \node[anchor=center] at (4,-2.5) {index = 2};
                            \draw[dashed] (0.5, 1)--(0.5, -1);
                            \draw[dashed] (5.5, 1)--(5.5,-1);
                            \draw (4,1.5)--(4,-1.5);
                            \draw[gray] (6.5,0) circle (0.5);
                            \draw[red] (1,0) arc (-180:0:2);
                        \end{tikzpicture}
                    };
                    \node[subfig] (H) at (A |- E){
                        \begin{tikzpicture}
                            \foreach \x in {1,...,6}
                                \draw[fill] (\x,0) circle (\r);
                            \node[anchor=center] at (4,-2.5) {index = 3};
                            \draw[gray] (5.5,0) circle (0.5);
                            \draw[gray] (1,0) arc (-180:0:1.5);
                            \draw[gray] (2,0) arc (180:0:1);
                            \draw[gray] (2,0) arc (-180:0:0.5);
                        \end{tikzpicture}
                    };
                    \draw[->] ($(A.east)+(\d,0)$)--node[above, font=\Huge]{$1$}($(B.west)-(\d,0)$);
                    \draw[->] ($(B.east)+(\d,0)$)--node[above, font=\Huge]{$2$}($(C.west)-(\d,0)$);
                    \draw[->] ($(C.east)+(\d,0)$)--node[above, font=\Huge]{$3$}($(D.west)-(\d,0)$);
                    \draw[->] ($(D.south)-(0,\d)$)--node[right, font=\Huge]{$2$}($(E.north)+(0,\d)$);
                    \draw[->] ($(E.west)-(\d,0)$)--node[above, font=\Huge]{$3$}($(F.east)+(\d,0)$);
                    \draw[->] ($(F.west)-(\d,0)$)--node[above, font=\Huge]{$3$}($(G.east)+(\d,0)$);
                    \draw[->] ($(G.west)-(\d,0)$)--node[above, font=\Huge]{$\#$}($(H.east)+(\d,0)$);
                \end{tikzpicture}
            }
        \end{center}
        \caption{Index computation for $\lambda = (1, 3, 2)$, $\mu = C_6(-2, 2)$. Each color corresponds to a connected component of the current state, and the vertices between the dashed lines are those which are stored in the current state. The solid line in the middle separates vertices on the left and right sides of the seaweed meander.}
        \label{fig:alg-large-parts}
    \end{figure}
\end{ex}

We now check that this algorithm only tracks a bounded number of vertices in the state at any time.

For any vertex not in the large part of $\mu$, its edges are entirely determined by the parts of $\lambda$ and $\mu$ immediately around it. In particular, since $J$ is finite, these parts are bounded, so the algorithm only needs to keep track of some constant number of vertices at any point in order to identify the edges adjacent to these vertices.

For any vertex $v \in [a, n - b]$, its edge for $\lambda$ depends only on the (small) part which contains it in $\lambda$, so that edge may be treated in the same way as in the algorithm from~\cref{subsec:b0}. However, its edge for $\mu$ comes from the large part of $\mu$, and connects it to $n + a - b - v$.

In this case, note that $v$ and $n + a - b - v$ are approximately symmetric with respect to $\frac{n}{2}$, up to a constant offset in terms of $a$ and $b$.

Without loss of generality, say $v < n + a - b - v$. For all $v$ such that
\[v < \frac{n}{2} < n + a - b - v,\]
vertex $v$ corresponds to the $(2v - 1)$th input character, and $n + a - b - v$ corresponds to the $(2b + 2v - 2a + 2)$th input character, and these differ by a constant. Thus, in order to add the edge between these two vertices, we only need to track a bounded number of vertices.

For any other $v \in [a, n - b]$, both $v$ and $n + a - b - v$ must be within $\abs{a - b}$ of $\frac{n}{2}$, in which case both $v$ and $n + a - b - v$ are within $2\abs{a - b} + 2$ characters of the end-of-input character $\#$. Hence, we can match these pairs as well using only a bounded number of vertices.
Thus, in order to reconstruct the arcs of the seaweed meander, it is enough to store a bounded amount of information, and thus this algorithm has finitely many states.

Using this algorithm, we can now show that $B_1 ^ J$ is rational.

\begin{proof}[Proof of~\cref{prop:b1-rational}]
    We use essentially the same method as in the case where $J \subseteq \NN$. For fixed $s \in \calS_J$ and $\mu = C_n(s)$, the above algorithm corresponds to a finite automaton which computes the index of a seaweed meander $\frac{\lambda}{\mu}$. Furthermore, the valid words of length $n + 1$ correspond exactly to seaweed meanders of size $n$ (recall that we append the character $\#$ at the end of the input).

    Now, if we assign $y ^ i$ to an edge whenever the index increments by $i$, the weight of the walk corresponding to input $g(\lambda)$ is $y ^ {\ind\left(\frac{\lambda}{\mu}\right)}$. Thus, by~\cref{thm:stanley-gen-func} with $R = \QQ[y]$, the generating function
    \[\sum_{n \ge 0} \sum_\lambda x ^ {n + 1} y ^ {\ind\left(\frac{\lambda}{C_n(s)}\right)}\] is rational, where the inner sum is over all compositions $\lambda \vDash n$ with parts in $J ^ +$.

    Now dividing the above generating function by $x$ and truncating a leading polynomial term, we find that
    \[B_1 ^ J(x, y) = 2\sum_{s \in \calS_J} \sum_{n \ge N_0} \sum_\lambda x ^ n y ^ {\ind\left(\frac{\lambda}{C_n(s)}\right)} \in \QQ(x, y),\] where the factor of $2$ accounts for whether the large part appears in $\lambda$ or in $\mu$. Note that we again use the fact that $\calS_J$ is finite for rationality.
\end{proof}

\subsection{Rationality of $B_2 ^ J$}

We now consider the case where $\lambda$ and $\mu$ each have exactly one large part. Unlike the previous two cases, we can prove the rationality directly, without resorting to an automaton-based argument.

\begin{prop}\label{prop:big-parts-periodic}
    Fix $s_1, s_2 \in \calS_J$. Then exactly one of the following holds.
    \begin{enumerate}[label=(\roman*)]
        \item For all sufficiently large $n$,
        \[\ind\left(\frac{C_{n + 1}(s_1)}{C_{n + 1}(s_2)}\right) = \ind\left(\frac{C_n(s_1)}{C_n(s_2)}\right) + 1.\]
        \item There exists some integer $N > 0$ such that
        \[\ind\left(\frac{C_{n + N}(s_1)}{C_{n + N}(s_2)}\right) = \ind\left(\frac{C_n(s_1)}{C_n(s_2)}\right)\] for all sufficiently large $n$.
    \end{enumerate}
\end{prop}

\cref{ex:case1,ex:case2} motivate the proof. We first give an example of the situation in (i) in~\cref{fig:case-i}.

\begin{ex}\label{ex:case1}
    Suppose $s_1 = (1, -2, 1)$, and $s_2 = (1, 1, -4, 2)$. The seaweed meanders for $n = 7$ and $n = 8$ are shown in~\cref{fig:case-i}. Note that the two large parts share a midpoint, so increasing $n$ simply increases the number of nested cycles in the central region.
    \begin{figure}[!ht]
        \begin{center}
            \resizebox{0.6\textwidth}{!}{%
                \begin{tikzpicture}[every node/.style={font=\ttfamily\Huge},every path/.style={draw, very thick},subfig/.style={}]
                    \node[subfig, anchor=east](A) at (0,0){
                        \begin{tikzpicture}
                            \foreach \x in {1,...,7}
                                \draw[fill] (\x,0) circle (0.15);
                            \foreach \x in {1,2}
                                \draw (\x+1,0) arc (180:0:3-\x);
                            \foreach \x in {2}
                                \draw (\x+1,0) arc (-180:0:3-\x);
                            \draw (6,0) arc (-180:0:0.5);
                        \end{tikzpicture}
                    };
                    \draw[dashed] (1.5,2)--(1.5,-2);
                    \node[subfig, anchor=west] (B) at (3,0){
                        \begin{tikzpicture}
                            \foreach \x in {1,...,8}
                                \draw[fill] (\x,0) circle (0.15);
                            \foreach \x in {1,2,3}
                                \draw (\x+1,0) arc (180:0:3.5-\x);
                            \foreach \x in {2,3}
                                \draw (\x+1,0) arc (-180:0:3.5-\x);
                            \draw (7,0) arc (-180:0:0.5);
                        \end{tikzpicture}
                    };
                \end{tikzpicture}
        }
        \end{center}
        \caption{The seaweed meanders associated to $s_1 = (1, -2, 1)$ and $s_2 = (1, 1, -4, 2)$, with $n = 7$ on the left and $n = 8$ on the right.}
        \label{fig:case-i}
    \end{figure}
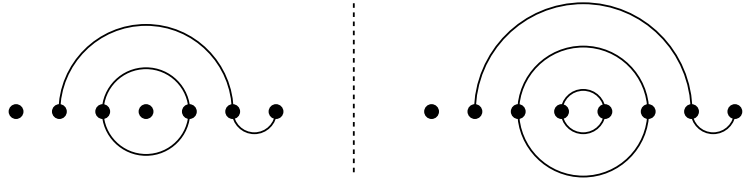
\end{ex}
We give an example of the situation in (ii) in~\cref{fig:case-ii}.

\begin{ex}\label{ex:case2}
    Take $s_1 = (0)$ and $s_2 = (1, -1)$, so that
    \[\frac{C_n(s_1)}{C_n(s_2)} = \frac{n}{1 \mid n - 1}.\]
    We may then take $N = 2$, as shown in~\cref{fig:case-ii}. Note that $N = 1$ works as well, but the proof below gives the value $N = 2$.
    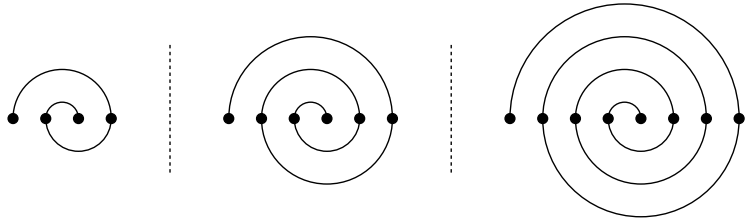
\begin{figure}[!ht]
        \begin{center}
            \resizebox{0.6\textwidth}{!}{%
                \begin{tikzpicture}[every node/.style={font=\ttfamily\Huge},every path/.style={draw, very thick},subfig/.style={}]
                    \node[subfig] (A) at (0, 0){
                        \begin{tikzpicture}
                            \foreach \x in {1,...,4}
                                \draw[fill] (\x, 0) circle (0.15);
                            \foreach \x in {1,2}
                                \draw (\x, 0) arc (180:0:2.5-\x);
                            \foreach \x in {2}
                                \draw (\x,0) arc (-180:0:3-\x);
                        \end{tikzpicture}
                    };
                    \node[subfig, anchor=west] (B) at ($(A.east) + (3, 0)$){
                        \begin{tikzpicture}
                            \foreach \x in {1,...,6}
                                \draw[fill] (\x, 0) circle (0.15);
                            \foreach \x in {1,2,3}
                                \draw (\x, 0) arc (180:0:3.5-\x);
                            \foreach \x in {2,3}
                                \draw (\x,0) arc (-180:0:4-\x);
                        \end{tikzpicture}
                    };
                    \node[subfig, anchor=west] (C) at ($(B.east) + (3, 0)$){
                        \begin{tikzpicture}
                            \foreach \x in {1,...,8}
                                \draw[fill] (\x, 0) circle (0.15);
                            \foreach \x in {1,2,3,4}
                                \draw (\x, 0) arc (180:0:4.5-\x);
                            \foreach \x in {2,3,4}
                                \draw (\x,0) arc (-180:0:5-\x);
                        \end{tikzpicture}
                    };
                    \draw[dashed] ($0.5*(A.east) + 0.5*(B.west) + (0, 2)$)--($0.5*(A.east) + 0.5*(B.west) + (0, -2)$);
                    \draw[dashed] ($0.5*(B.east) + 0.5*(C.west) + (0, 2)$)--($0.5*(B.east) + 0.5*(C.west) + (0, -2)$);
                \end{tikzpicture}
        }
        \end{center}
        \caption{The seaweed meanders $\frac{n}{1 \mid n - 1}$ for $n = 4, 6, 8$, from left to right.}
        \label{fig:case-ii}
    \end{figure}
\end{ex}
In fact, as we will show in the proof, in case (ii) the number of paths and cycles are preserved when we go from $n$ to $n + N$, because the components of $\frac{C_n(s_1)}{C_n(s_2)}$ will be contractions of the components of $\frac{C_{n + N}(s_1)}{C_{n + N}(s_2)}$.

\begin{proof}[Proof of~\cref{prop:big-parts-periodic}]
    Take $n$ large. Suppose the large part of $C_n(s_i)$ covers the interval $[a_i, n - b_i]$. Note that $a_i$ and $b_i$ do not depend on $n$.
    
    Let
    \[c_i = \frac{n - b_i + a_i}{2}\] be the center of the large part in $C_n(s_i)$.

    We split into two cases, depending on whether $c_1 = c_2$, which is equivalent to $a_1 - b_1 = a_2 - b_2$. The case where $c_1 = c_2$ corresponds to (i), and $c_1 \ne c_2$ corresponds to (ii).

    First, suppose $c_1 = c_2 = c$. Note that since $n$ is large, each edge $(v, w)$ in the seaweed meander $M_n = \frac{C_n(s_1)}{C_n(s_2)}$ where $v < c < w$ comes from a large part. Furthermore, this edge becomes the edge $(v, w + 1)$ in $M_{n + 1} = \frac{C_{n + 1}(s_1)}{C_{n + 1}(s_2)}$, since the centers of the large parts increase by $\half$ when $n$ increases by $1$.

    By shifting all vertices greater than $c$ in $M_n$ up by $1$ in $M_{n + 1}$, we get a bijection between connected components in the two seaweed meanders not involving the vertices $c$ and $c + \half$.
    
    Now, if $c \in \ZZ$, then the isolated vertex at $c$ in $M_n$ becomes a cycle on the vertices $c$ and $c + 1$ in $M_{n + 1}$, while if $c \in \ZZ + \half$, then a new isolated vertex at $c + \half$ appears in $M_{n + 1}$.
    In either case, the index increases by $1$, as desired.

    Now, suppose $c_1 \ne c_2$. Without loss of generality, say $c_1 < c_2$. Let $\ell = c_2 - c_1$. We claim the second case of the proposition holds with $N = 4\ell$.
    As above, we have the seaweed meanders
    \begin{align*}
        M_n &= \frac{C_n(s_1)}{C_n(s_2)}, \\
        M_{n + 4\ell} &= \frac{C_{n + 4\ell}(s_1)}{C_{n + 4\ell}(s_2)}.
    \end{align*}
    We again define a map between vertices in $M_n$ and $M_{n + 4\ell}$ which preserves connected components.
    Fix some $m < \frac{n}{2}$ which is sufficiently far from each of the $a_i$, $b_i$, and $c_i$, for example, $m = \floor{\frac{n}{3}}$.

    Define
    \[f(v) = \begin{cases}
         v & v \le m, \\
         v + 2\ell & m < v < 2c_1 - m, \\
         v + 4\ell & 2c_1 - m \le v.
    \end{cases}\]
    We claim that $f(v)$ and $f(w)$ are in the same connected component in $M_{n + 4\ell}$ if and only if $v$ and $w$ are in the same connected component in $M_n$.

    In $M_{n + 4\ell}$, the large parts are centered at $c_1 + 2\ell$ and $c_2 + 2\ell$, so the edges of the large parts are of form $(v, 2c_i + 4\ell - v)$, for $i = 1, 2$.

    Consider any edge $(v, w)$ in $M_n$, with $v < w$. We have several cases.

    If $(v, w)$ is an edge from a small part, then either $v$ and $w$ both lie in $[1, m]$ or both lie in $[2c_1 - m, n]$. In either case, $v$ and $w$ are shifted by the same amount, so $(f(v), f(w))$ is again an edge for the same small part.

    Now suppose $(v, w)$ is an edge from a large part, so $v + w = 2c_i$ for some $i$. Note that if $v + w = 2c_1$, we always have $f(v) + f(w) = 2c_1 + 4\ell$, so $(f(v), f(w))$ is again an edge in $M_{n + 4\ell}$.

    It remains to consider when $v + w = 2c_2$. If $v \le m$, then $w = 2c_2 - v \ge 2c_1 - m$, so $f(v) + f(w) = 2c_2 + 4\ell$. If $v > m + 2\ell$, then
    \[w < 2c_2 - m - 2\ell = 2c_1 - m,\] so $f(v) = v + 2\ell$ and $f(w) = w + 2\ell$, and we again have $f(v) + f(w) = 2c_2 + 4\ell$. In either case, $(f(v), f(w))$ is an edge in $M_{n + 4\ell}$.

    Finally, suppose $m < v \le m + 2\ell$. Then
    \[w = 2c_2 - v \ge 2c_2 - m - 2\ell = 2c_1 - m,\] so $f(v) = v + 2\ell$ and $f(w) = w + 4\ell$. In $M_{n + 4\ell}$, we have the edges
    \[(v + 2\ell, w + 2\ell), (w + 2\ell, v), (v, w + 4\ell).\] The edge $(w + 2\ell, v)$ is in $M_{n + 4\ell}$ since
    \[v + w + 2 \ell = 2c_2 + 2\ell = 2c_1 + 4\ell.\]
    Hence, $f(v)$ and $f(w)$ are connected in $M_{n + 4\ell}$.

    Note that the paths in the last case use all edges which involve vertices from
    \[[n + 4\ell] \setminus f([n]) = [m + 1, m + 2\ell] \cup [2c_1 + 2\ell - m, 2c_1 + 4\ell - m - 1],\] since all edges adjacent to these vertices come from large parts. Thus, all edges of $M_{n + 4\ell}$ are either accounted for by these paths or are images of edges from $M_n$.

    We now check that each path or cycle in $M_n$ corresponds to a path or cycle, respectively, in $M_{n + 4\ell}$. The above casework shows that vertices in $M_{n + 4\ell}$ not in $f([n])$ all have degree $2$, and that the map $f$ preserves degree. Hence, the endpoints of any path in $M_n$ are sent to endpoints of paths in $M_{n + 4\ell}$ by $f$, while degree $2$ vertices are sent to degree $2$ vertices.

    Thus, no new connected components are formed when we go from $M_n$ to $M_{n + 4\ell}$, and paths/cycles are unchanged.
    Hence, $\ind(M_n) = \ind(M_{n + 4\ell})$.
\end{proof}

In~\cref{ex:case2}, where $s_1 = (0)$ and $s_2 = (1, -1)$, we have $c_1 = \frac{n + 1}{2}$, and $c_2 = \frac{n + 2}{2}$. This gives $\ell = \half$, so $4\ell = 2$.

\begin{cor}\label{cor:b2}
    For finite $J \subseteq \ZZ$, the generating function $B_2 ^ J$ is rational.
\end{cor}

\begin{proof}
    Consider fixed $s_1, s_2 \in \calS_J$. Then define
    \[B_{s_1, s_2}(x, y) = \sum_{n \ge N_0} x ^ n y ^ {\ind\left(\frac{C_n(s_1)}{C_n(s_2)}\right)}.\]

    By~\cref{prop:big-parts-periodic}, there are two possibilities. In case (i), the terms of $B_{s_1, s_2}$ are eventually a geometric progression with common ratio $xy$, in which case $B_{s_1, s_2}$ is rational with denominator $1 - xy$.

    In case (ii), the exponents of $y$ are eventually periodic with period $N$, so $B_{s_1, s_2}$ is rational with denominator $1 - x ^ N$.

    Hence,
    \[B_2 ^ J = \sum_{s_1, s_2 \in \calS_J} B_{s_1, s_2}\] is also rational, since $\calS_J$ is finite.
\end{proof}

\subsection{Final Proof}

We now have all the preliminaries required to prove~\cref{thm:rational} for general $J$.

\begin{proof}[Proof of~\cref{thm:rational}]
    Fix $J$. By~\cref{prop:ab-equiv}, it suffices to show that $B ^ J$ is rational. Furthermore, recall that we have
    \[B ^ J = B_0 ^ J + B_1 ^ J + B_2 ^ J.\]

    By combining~\cref{lem:ab-positive,prop:rational-positive}, we have that $B_0 ^ J$ is rational.
    By~\cref{prop:b1-rational}, $B_1 ^ J$ is rational, and by~\cref{cor:b2}, $B_2 ^ J$ is rational.
    Thus, $B ^ J = B_0 ^ J + B_1 ^ J + B_2 ^ J$ is rational, so $A ^ J$ is rational as well.
\end{proof}

\subsection{Explicit Generating Functions}

We now state the generating functions $A ^ {\{-1, 1, 2\}}$, $A ^ {\{0, 1, 3\}}$, and $A ^ {\{1, 6\}}$, answering a question of Archer, Geary, and Laudone~\cite{archer2025}. The proofs are quite long and mostly consist of computations, so we defer them to~\cref{appendix:bashing}.

\begin{theorem}\label{thm:-112-gen-func}
    We have
    \begin{align*}
        A ^ {\{-1, 1, 2\}} &= \frac{1 - x}{(1 - xy - x ^ 2y ^ 2)(1 - x) - 2x ^ 2y} + \frac{2x(1 + 2y + x ^ 2y - x ^ 4y ^ 2)(1 + xy - x ^ 2 + x ^ 2y ^ 2 + x ^ 3y - x ^ 4y ^ 2)}{1 - x ^ 2 - x ^ 2y - x ^ 4y - x ^ 4y ^ 2 + x ^ 6y ^ 2} \\
        &\phantom{=} - 2x(1 + 2y + xy + 2xy ^ 2 + 2x ^ 2y + 3x ^ 2y ^ 2 + 2x ^ 2y ^ 3) + \frac{2x ^ 4y ^ 4}{1 - xy} + \frac{2(x ^ 4y ^ 2 + x ^ 5y)}{1 - x ^ 2}.
    \end{align*}
\end{theorem}

\begin{theorem}\label{thm:013-gen-func}
    We have
    \begin{align*}
        A ^ {\{0, 1, 3\}}(x, y) &= \frac{1 - x ^ 2y - x ^ 3y}{(1 - x ^ 2y - x ^ 3y)(1 - xy - x ^ 3y ^ 3) - 2x ^ 3y ^ 2(1 + x)} \\
        &\phantom{=} + \frac{(2 + 2xy + 2x ^ 3y ^ 3)(1 - x ^ 4y - x ^ 6y) + 4x ^ 4y + 4x ^ 5y ^ 2(1 + x ^ 2)}{(1 - x ^ 4y - x ^ 6y)(1 - x ^ 2y - x ^ 6y ^ 3) - 2x ^ 6y ^ 2(1 + x ^ 2)} \\
        &\phantom{=} -2 - 2xy + x ^ 2y ^ 2 - 2x ^ 3y ^ 2 - 2x ^ 3y ^ 3 + \frac{x ^ 4y ^ 4}{1 - xy}.
    \end{align*}
\end{theorem}

\begin{theorem}\label{thm:16-gen-func}
    We have $A ^ {\{1, 6\}}(x, y) = \frac{f(x, y)}{g(x, y)}$, where
    \begin{align*}
        f(x, y) &= 1 - x ^ 3 - 2x ^ 4y - x ^ 5y ^ 2 - 2x ^ 6 - x ^ 6y ^ 2 + x ^ 8y ^ 2 + 2x ^ 9 + x ^ 9y ^ 2 + x ^ 9y ^ 3 + 4x ^ {10}y + x ^ {10}y ^ 3 + x ^ {11}y ^ 2 \\
        &\phantom{=} + x ^ {11}y ^ 4 + x ^ {12} + 2x ^ {12}y ^ 2 - x ^ {13}y - 2x ^ {14} y ^ 2 - x ^ {15} - 2x ^ {15}y ^ 2 - x ^ {16}y - 2x ^ {16}y ^ 3 + x ^ {17}y ^ 2 - x ^ {17}y ^ 4 \\
        &\phantom{=} - x ^ {18}y ^ 2 + x ^ {19}y + x ^ {21}y ^ 2 - x ^ {22}y + x ^ {22}y ^ 3\\
        \intertext{and}
        g(x, y) &= 1 - xy - x ^ 3 - x ^ 4y + x ^ 5y ^ 2 - 2x ^ 6 - x ^ 6y ^ 2 - x ^ 6y ^ 3 - x ^ 6y ^ 6 + x ^ 7y ^ 3 - x ^ 8y ^ 2 + 2x ^ 9 + x ^ 9y ^ 2 + x ^ 9y ^ 6 \\
        &\phantom{=} + 4x ^ {10}y + x ^ {10}y ^ 4 + 2x ^ {10}y ^ 7 + 3x ^ {11}y ^ 2 + x ^ {11}y ^ 8 + x ^ {12} + 2x ^ {12}y ^ 2 + x ^ {12}y ^ 3 + x ^ {12}y ^ 5 + 2x ^ {12} y ^ 6 + x ^ {12} y ^ 8\\
        &\phantom{=}  - x ^ {14}y ^ 2 - x ^ {14}y ^ 8 - x ^ {15} - 2x ^ {15}y ^ 2 - 2x ^ {15} y ^ 3 - 2x ^ {15} y ^ 6 - x ^ {15} y ^ 8 - x ^ {15} y ^ 9 - 2x ^ {16} y - 2x ^ {16} y ^ 3 - 4x ^ {16} y ^ 7 \\ 
        &\phantom{=} - x ^ {16} y ^ 9 - 2x ^ {17} y ^ 2 - x ^ {17} y ^ 4 - x ^ {17} y ^ 8 - x ^ {17} y ^ {10} - x ^ {18} y ^ 2 + x ^ {18} y ^ 3 - x ^ {18} y ^ 5 - x ^ {18} y ^ 6 - 2x ^ {18} y ^ 8 + x ^ {19} y \\
        &\phantom{=} - x ^ {19} y ^ 3 + x ^ {19}y ^ 7  + x ^ {20} y ^ 2 + 2x ^ {20} y ^ 8 + x ^ {21} y ^ 2 + x ^ {21} y ^ 6 + 2x ^ {21} y ^ 8 - x ^ {22} y + 2x ^ {22} y ^ 3 + x ^ {22} y ^ 7 + 2x ^ {22} y ^ 9 \\
        &\phantom{=} - x ^ {23} y ^ 2 + x ^ {23} y ^ 4 - x ^ {23} y ^ 8 + x ^ {23} y ^ {10} + x ^ {24} y ^ 8 - x ^ {25} y ^ 7 - x ^ {27} y ^ 8 + x ^ {28} y ^ 7 - x ^ {28} y ^ 9.
    \end{align*}
\end{theorem}
\section{Billiards}\label{sec:billiards}

Eu, Fu, and Hsu~\cite[Theorem 1.1(i)]{eu2025} give a connection between meandric systems and billiard trajectories in polygons on square lattices. We prove an analogous result which gives a similar connection between certain seaweed meanders and billiard trajectories in polygons on triangular lattices. We will also show that the restriction to seaweed meanders is necessary, as the natural corresponding statement for meandric systems does not hold.

Take $\lambda, \mu \vDash n$. Note that since all parts of $2\lambda$ and $2\mu$ are even, both sides of the seaweed meander $\frac{2\lambda}{2\mu}$ consist of complete non-crossing arc diagrams, and thus the seaweed meander is a union of cycles.

Under the standard correspondence between non-crossing arc diagrams and Dyck paths (where an arc to the right is an up-step and an arc to the left is a down-step), each side of the seaweed meander corresponds to a Dyck path of length $2n$. 

\begin{definition}\label{def:meander-to-shape}
    Let $M$ be a meandric system with non-crossing arc diagrams above and below the central line. Place the corresponding Dyck paths for the arc diagrams above and below the $x$-axis, and let $\rect(M)$ be the chessboard polygon they bound. Likewise, define $\tri(M)$ to be the corresponding shape in the triangular grid.
\end{definition}

In~\cref{fig:correspondence-ex}, we give an example of the correspondences defined in~\cref{def:meander-to-shape}, which we use as a running example throughout this section.
The first row shows the correspondence for the seaweed meander $\frac{6 \mid 4} {2 \mid 6 \mid 2}$, and the second row shows the correspondence for a meandric system.
\begin{figure}[!ht]
    \begin{center}
        \resizebox{\textwidth}{!}{%
            \begin{tikzpicture}[every node/.style={font=\ttfamily\Huge},every path/.style={draw, very thick},subfig/.style={}]
                \node[subfig, anchor=east] (A) at (0,0){
                    \begin{tikzpicture}
                        \foreach \x in {1,...,10}
                            \draw[fill] (\x, 0) circle (0.15);
                        \foreach \x in {1,2,3}
                        {
                            \draw (\x, 0) arc (180:0:3.5-\x);
                            \draw (\x+2,0) arc (-180:0:3.5-\x);
                        }
                        \draw (7,0) arc (180:0:1.5);
                        \draw (8,0) arc (180:0:0.5);
                        \draw (1,0) arc (-180:0:0.5);
                        \draw (9,0) arc (-180:0:0.5);
                    \end{tikzpicture}
                };
                \node[subfig, anchor=west] (B) at ($(A.east) + (2,0)$){
                    \begin{tikzpicture}
                        \draw (0,0)-- ++(1, 1) -- ++(1, 1) -- ++(1, 1) -- ++(1, -1)-- ++(1, -1) -- ++(1, -1) -- ++(1, 1) -- ++(1, 1) -- ++(1, -1) -- ++(1, -1);
                        \draw (0,0) -- ++(1, -1) -- ++(1, 1) -- ++(1, -1) -- ++(1, -1) -- ++(1, -1) -- ++(1, 1) -- ++(1, 1) -- ++(1, 1) -- ++(1, -1) -- ++(1, 1);
                        \draw[dashed] (0,0)--(10,0);
                        \draw[gray, dotted] (1,1)--(2,0)--(4,2);
                        \draw[gray, dotted] (2,2)--(6,-2);
                        \draw[gray, dotted] (3,-1)--(5,1);
                        \draw[gray, dotted] (4,-2)--(6,0)--(7,-1);
                        \draw[gray, dotted] (7,1)--(8,0)--(9,1);
                        \draw[red, dashed] (0.5,0.5)--(5.5,0.5)--(5.5,-2.5)--(4.5,-2.5)--(4.5,1.5)--(1.5,1.5)--(1.5,-0.5)--(0.5,-0.5)--(0.5,0.5);
                        \draw[blue, dashed] (2.5,2.5)--(3.5,2.5)--(3.5,-1.5)--(6.5,-1.5)--(6.5,0.5)--(9.5,0.5)--(9.5,-0.5)--(8.5,-0.5)--(8.5,1.5)--(7.5,1.5)--(7.5,-0.5)--(2.5,-0.5)--(2.5,2.5);
                    \end{tikzpicture}
                };
                \node[subfig, anchor=west] (C) at ($(B.east) + (2,0)$){
                    \begin{tikzpicture}
                        \draw (0,0)-- ++(60:2) -- ++(60:2) -- ++(60:2) -- ++(-60:2)-- ++(-60:2) -- ++(-60:2) -- ++(60:2) -- ++(60:2) -- ++(-60:2) -- ++(-60:2);
                        \draw (0,0) -- ++(-60:2) -- ++(60:2) -- ++(-60:2) -- ++(-60:2) -- ++(-60:2) -- ++(60:2) -- ++(60:2) -- ++(60:2) -- ++(-60:2) -- ++(60:2);
                        \draw[dashed] (0,0)--(10,0);
                        \draw[red, dashed] (60:1) -- ++(5,0) -- ++(-120:4) -- ++(3,0) -- ++(120:8) -- ++(1,0) -- ++(-120:6) -- ++(1,0) -- ++(120:2);
                        \draw[blue, dashed] (60:3) -- ++(3,0) -- ++(-120:4) -- ++(5,0) -- ++(120:2) -- ++(3,0) -- ++(-120:2) -- ++(1, 0) -- ++(120:4) -- ++(1,0) -- ++(-120:8) -- ++(1, 0) -- ++(120:8);
                        \draw[gray, dotted] (60:2) -- ++(4, 0) -- ++(-120:4) -- ++(4, 0);
                        \draw[gray, dotted] (60:4) -- ++(2, 0) -- ++(-120:4);
                        \draw[gray, dotted] (6,0)++(60:2)-- ++(2,0) -- ++(-120:2) -- ++(120:2);
                        \draw[gray, dotted] (2,0)++(-60:4)-- ++(2,0) -- ++(120:8);
                        \draw[gray, dotted] (60:2)--(2,0);
                        \draw[gray, dotted] (6,0)-- ++(-120:4);
                        \draw[gray, dotted] (6,0)-- ++(-60:2);
                    \end{tikzpicture}
                };
                \node (left) at ($(A.west |- C.south) + (-0.5, -0.5)$) {};
                \node (right) at ($(C.east |- C.south) + (0.5, -0.5)$) {};
                \draw[dashed] (left)--(right);
                \node[subfig, anchor=north] (F) at ($(C.south) + (0, -1)$){
                    \begin{tikzpicture}
                        \draw (0, 0) -- ++(60:2) -- ++(60:2) -- ++(60:2) -- ++(-60:2) -- ++(-60:2) -- ++(60:2) -- ++(-60:2) -- ++(-60:2);
                        \draw (0, 0) -- ++(-60:2) -- ++(-60:2) -- ++(60:2) -- ++(-60:2) -- ++(60:2) -- ++(-60:2) -- ++(60:2) -- ++(60:2);
                        \draw[dashed] (0,0)--(8,0);
                        \draw[orange, dashed] (60:1) -- ++(7, 0) -- ++(-120:4) -- ++(1,0) -- ++(120:8) -- ++(1,0) -- ++(-120:6) -- ++(7,0) -- ++(120:4) -- ++(1,0) -- ++(-120:6) -- ++(1,0) -- ++(120:6) -- ++ (3, 0) -- ++(-120:6) -- ++(1,0) -- ++(120:4);
                        \draw[gray, dotted] (60:2) -- ++ (6,0) -- ++(-120:4) -- ++(120:6) -- ++(2,0) -- ++(-120:6) -- ++(6, 0) -- ++(120:4) -- ++(-120:4) -- ++(120:4);
                    \end{tikzpicture}
                };
                \node[subfig] (D) at (A |- F){
                    \begin{tikzpicture}
                        \foreach \x in {1,...,8}
                            \draw[fill] (\x,0) circle (0.15);
                        \draw (1, 0) arc (180:0:3.5);
                        \draw (2, 0) arc (180:0:1.5);
                        \draw (3, 0) arc (180:0:0.5);
                        \draw (6, 0) arc (180:0:0.5);
                        \draw (1,0) arc (-180:0:3.5);
                        \foreach \x in {1,2,3}
                            \draw (2*\x, 0) arc (-180:0:0.5);
                    \end{tikzpicture}
                };
                \node[subfig] (E) at (B |- F){
                    \begin{tikzpicture}
                        \draw (0, 0) -- ++(1, 1) -- ++(1, 1) -- ++(1, 1) -- ++(1, -1) -- ++ (1, -1) -- ++(1, 1) -- ++ (1, -1) -- ++(1, -1);
                        \draw (0, 0) -- ++(1, -1) -- ++(1, -1) -- ++(1, 1) -- ++(1, -1) -- ++(1, 1) -- ++(1, -1) -- ++(1, 1) -- ++(1, +1);
                        \draw[dashed] (0,0)--(8,0);
                        \draw[gray, dotted] (1,-1)--(4,2);
                        \draw[gray, dotted] (1, 1)--(3,-1)--(5,1)--(7,-1);
                        \draw[gray, dotted] (2, 2)--(5,-1)--(7,1);
                        \draw[red, dashed] (0.5,0.5)--(0.5,-0.5)--(7.5,-0.5)--(7.5,0.5)--(0.5,0.5);
                        \draw[blue, dashed] (1.5,1.5)--(1.5,-1.5)--(2.5,-1.5)--(2.5,2.5)--(3.5,2.5)--(3.5,-1.5)--(4.5,-1.5)--(4.5,1.5)--(1.5,1.5);
                        \draw[green,dashed] (5.5,1.5)--(6.5,1.5)--(6.5,-1.5)--(5.5,-1.5)--(5.5,1.5);
                    \end{tikzpicture}
                };
            \end{tikzpicture}
    }
    \end{center}
    \caption{The correspondences given in~\cref{def:meander-to-shape}. In each row, we have $M$, $\rect(M)$, and $\tri(M)$, from left to right. 
    The lines which form the lattices are drawn with dotted lines. The billiard trajectories in each shape are drawn with dashed lines, with each color representing a different trajectory.}
    \label{fig:correspondence-ex}
\end{figure}
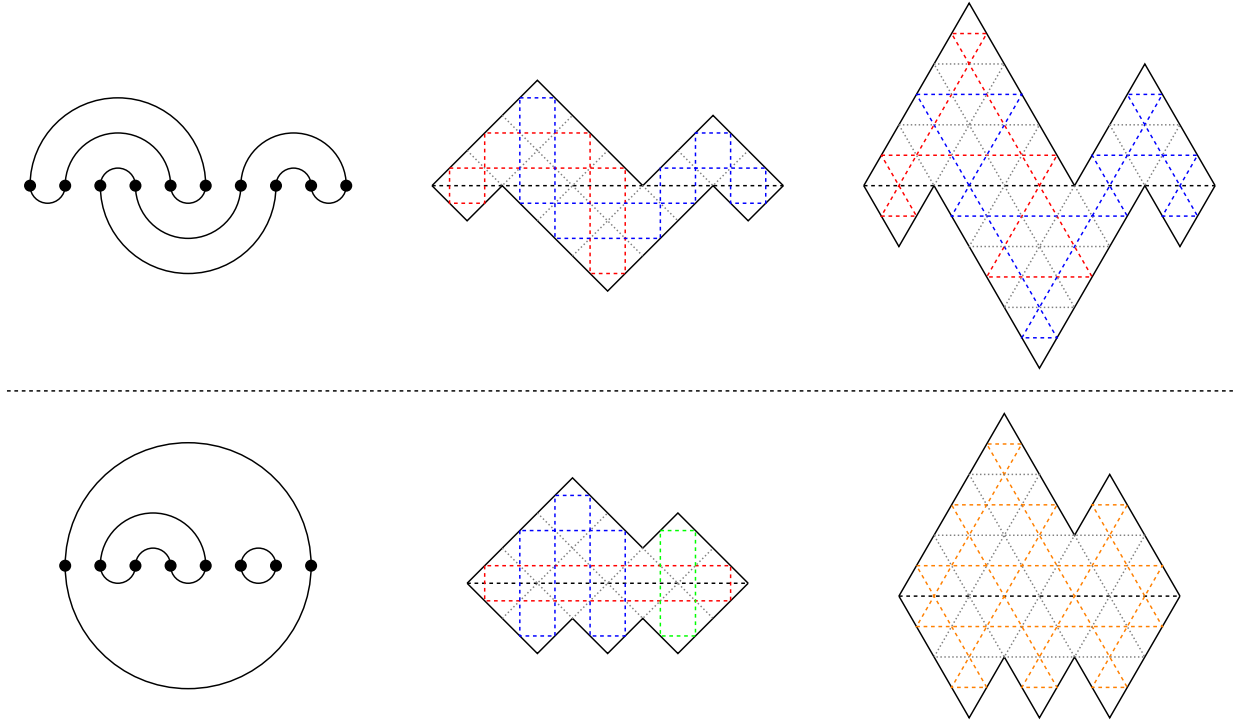

We now state the result by Eu, Fu, and Hsu~\cite{eu2025} giving a correspondence between meandric systems and billiard trajectories.
\begin{theorem}[{\cite[Theorem 1.1(i)]{eu2025}}]\label{thm:rect-traj}
    For any meandric system $M$, the number of connected components of $M$ is equal to the number of billiard trajectories in $\rect(M)$.
\end{theorem}

As shown by Eu, Fu, and Hsu~\cite{eu2025}, this follows directly from viewing the horizontal segments of the billiard trajectories in $\rect(M)$ as the arcs of the arc diagrams.

\begin{definition}
    Let the \textit{length} of a billiard trajectory be the number of reflections it takes.
\end{definition}

The same proof as in~\cite{eu2025} gives the following result about the lengths of the billiard trajectories.

\begin{cor}\label{cor:traj-lengths}
    For a meandric system $M$, the lengths of the billiard trajectories in $\rect(M)$ are exactly twice the sizes of their corresponding connected components in $M$.
\end{cor}

We now state our main result on billiard trajectories in triangular shapes.

\begin{theorem}\label{thm:triangle-billiard}
    For a seaweed meander $M = \frac{2\lambda}{2\mu}$, there is a bijection between connected components of $M$ and billiard trajectories in $\tri(M)$. Furthermore, the length of each billiard trajectory is twice the size of its corresponding cycle.
\end{theorem}

Note that the analogous statement does not hold for meandric systems. In particular, in the bottom row of~\cref{fig:correspondence-ex}, $\rect(M)$ has three billiard trajectories, while $\tri(M)$ has only one.

\begin{proof}
    Suppose $\lambda, \mu \vDash n$, so that $M$ is a seaweed meander of size $2n$. Then the boundary of $\tri(M)$ consists of $4n$ unit segments, with $2n$ above the midline and $2n$ below. The midline consists of $n$ unit segments.

    Now, for each segment on the midline of $\tri(M)$, consider the two lines through its midpoint at angles $60\dg$ and $120\dg$ to the positive $x$-axis. Each line intersects exactly one unit boundary segment of $\tri(M)$ above and below the midline.
    
    For the $i$th segment, counting from the left, label the boundary unit segments which the $60 \dg$ line intersects with $2i$, and label the boundary unit segments intersecting the $120\dg$ line with $2i - 1$.

    An example is given in~\cref{fig:label-ex}, again using the seaweed meander $\frac{6 \mid 4}{2 \mid 6 \mid 2}$.
    \begin{figure}[!ht]
        \begin{center}
            \resizebox{0.45\textwidth}{!}{%
                \begin{tikzpicture}[every path/.style={draw, very thick}, every node/.style={font=\large}]
                    \draw (0,0)-- node[above left] {1} ++(60:2) -- node[above left] {3} ++(60:2) -- node[above left] {5} ++(60:2) -- node[above right] {2} ++(-60:2)-- node[above right] {4} ++(-60:2) -- node[above right] {6} ++(-60:2) -- node[above left] {7} ++(60:2) -- node[above left] {9} ++(60:2) -- node[above right] {8} ++(-60:2) -- node[above right] {10} ++(-60:2);
                    \draw (0,0) -- node[below left] {2} ++(-60:2) -- node[below right] {1} ++(60:2) -- node[below left] {4} ++(-60:2) -- node[below left] {6} ++(-60:2) -- node[below left] {8} ++(-60:2) -- node[below right] {3} ++(60:2) -- node[below right] {5} ++(60:2) -- node[below right] {7} ++(60:2) -- node[below left] {10} ++(-60:2) -- node[below right] {9} ++(60:2);
                    \draw[dashed] (0,0)--(10,0);
                    \draw[red, dashed] (60:1) -- ++(5,0) -- ++(-120:4) -- ++(3,0) -- ++(120:8) -- ++(1,0) -- ++(-120:6) -- ++(1,0) -- ++(120:2);
                    \draw[blue, dashed] (60:3) -- ++(3,0) -- ++(-120:4) -- ++(5,0) -- ++(120:2) -- ++(3,0) -- ++(-120:2) -- ++(1, 0) -- ++(120:4) -- ++(1,0) -- ++(-120:8) -- ++(1, 0) -- ++(120:8);
                    \draw[gray, dotted] (60:2) -- ++(4, 0) -- ++(-120:4) -- ++(4, 0);
                    \draw[gray, dotted] (60:4) -- ++(2, 0) -- ++(-120:4);
                    \draw[gray, dotted] (6,0)++(60:2)-- ++(2,0) -- ++(-120:2) -- ++(120:2);
                    \draw[gray, dotted] (2,0)++(-60:4)-- ++(2,0) -- ++(120:8);
                    \draw[gray, dotted] (60:2)--(2,0);
                    \draw[gray, dotted] (6,0)-- ++(-120:4);
                    \draw[gray, dotted] (6,0)-- ++(-60:2);
                \end{tikzpicture}
        }
        \end{center}
        \caption{Example edge labels for the seaweed meander $\frac{6 \mid 4}{2 \mid 6 \mid 2}$.}
        \label{fig:label-ex}
    \end{figure}
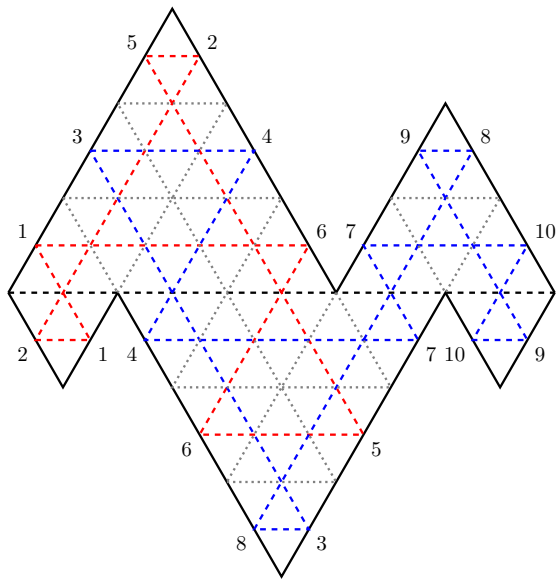

    We claim that the billiard trajectory segments parallel to the midline correspond exactly to the arcs of the seaweed meander. In particular, arcs of $2\lambda$ correspond to horizontal reflections above the midline, and arcs of $2\mu$ correspond to horizontal reflections below the midline.

    We prove this for arcs of $2\lambda$, and the argument for $2\mu$ is essentially the same. Consider a part of $2\lambda$ of size $2m$. Suppose this part corresponds to the interval $[2a + 1, 2a + 2m]$ in $M$. The boundary segments of $\tri(M)$ corresponding to this part then form an equilateral triangle with side length $m$ above the midline.

    Hence, the corresponding labels on the left edge (from bottom to top) are
    \[2a + 1, 2a + 3, \dots, 2a + 2m - 1,\] while the labels on the right edge (from bottom to top) are
    \[2a + 2m, 2a + 2m - 2, \dots, 2a + 2.\]
    Thus, each of the horizontal lines in this equilateral triangle corresponds to a pair of vertices in the interval with sum $4a + 2m + 1$. These are exactly the pairs of vertices in $M$ which are connected by arcs corresponding to the part $2m$.

    Note that each billiard trajectory in $\tri(M)$ alternates between segments parallel to the midline and segments at either a $60\dg$ or $120\dg$ angle to the positive $x$-axis. By construction, each segment not parallel to the midline connects unit boundary segments with the same label. Furthermore, we have already shown that the horizontal segments of the billiard trajectories correspond to the arcs of $M$.

    Now, suppose we follow a billiard trajectory and keep track of the labels on the boundary segments it intersects. Whenever we follow a segment in the trajectory which crosses the midline, we repeat a label. A horizontal segment above the midline corresponds to following an arc for $2\lambda$, and a horizontal segment below the midline corresponds to following an arc for $2\mu$.
    
    Hence, as we traverse the billiard trajectory, the labels on the corresponding boundary segments traverse a cycle in $M$, with each vertex repeated twice to account for the non-horizontal segments in the trajectories.
    Thus, the number of trajectories equals the number of connected components of $M$, and the length of each trajectory is twice the size of the corresponding component.
\end{proof}

As an example, in~\cref{fig:label-ex}, the red trajectory intersects boundary segments with labels
\[1, 1, 6, 6, 5, 5, 2, 2,\] while the corresponding seaweed meander (seen in the first row of~\cref{fig:correspondence-ex}) has a cycle consisting of vertices $1$, $6$, $5$, and $2$, in that order.

Using~\cref{thm:triangle-billiard}, we can also interpret the index of any seaweed meander in terms of the number of billiard trajectories in a triangular shape.

\begin{cor}\label{cor:index-traj}
    Let $M = \frac{\lambda}{\mu}$ be a seaweed meander, and define $M' = \frac{2\lambda}{2\mu}$. Then $\ind(M)$ is equal to the number of billiard trajectories in each of $\rect(M')$ and $\tri(M')$.
\end{cor}

\begin{proof}
    By~\cite[Proposition 2.3]{archer2025}, we have
    \[\ind(M') = \ind\left(\frac{2\lambda}{2\mu}\right) = 2\ind\left(\frac{\lambda}{\mu}\right) = 2\ind(M).\]
    Note that $M'$ consists of only cycles, so $\ind(M')$ is twice the number of connected components of $M'$.

    Thus, $\ind(M)$ is equal to the number of connected components in $M'$.
    By \cref{thm:rect-traj,thm:triangle-billiard} this is equal to the number of billiard trajectories in each of $\rect(M')$ and $\tri(M')$.
\end{proof}
\section{Further Questions}\label{sec:questions}

We pose a number of questions extending the above results.

\subsection{Central Limit Theorems}

There are several directions in which we may strengthen the probabilistic results given in~\cref{sec:distrib}.
\begin{question}
    Are there closed-form expressions for the constants $a$ and $b$ in~\cref{thm:clt-seaweed}?
\end{question}

As mentioned after the proof of~\cref{thm:clt-seaweed}, we have the estimates $a \approx 0.42$ and $b \approx 0.49$.

Recall that in the proof of~\cref{thm:clt-seaweed}, we give formulas for $a$ and $b$ in terms of the probability distribution $\rho$ on the set of indecomposable components. However, it seems more difficult to understand how the index behaves under this distribution. In particular, by~\cref{cor:more-clt} similar central limit theorems hold for other statistics, which suggests that there may not be any special property of the index which can be used to compute closed forms.

For another perspective on the central limit theorem of~\cref{thm:clt-seaweed}, note that both sides of a seaweed meander of size $n$ correspond to an involution in $S_n$. Furthermore, as shown by Dergachev and Kirillov~\cite[Theorem 3.3]{Dergachev2000}, the number of cycles in the product of these involutions is equal to the index of the seaweed meander.

Similarly, as shown by Karnauhova and Liebscher~\cite[Proposition 2.2]{karnauhova2015}, if we take the involutions corresponding to each side of a meandric system, the number of cycles in their product is equal to twice the number of cycles of the meandric system.

\begin{definition}
    A \textit{layered permutation} in $S_n$ is a permutation formed by partitioning $[n]$ into contiguous blocks and reversing each block (that is, concatenating the decreasing permutations on each block). Equivalently, they are exactly the permutations which may be realized by the arcs on one half of a seaweed meander.
\end{definition}

Under the above interpretation,~\cref{thm:clt-seaweed} gives a central limit theorem for the number of cycles in the product of two uniformly random layered permutations, while the conjecture of Kargin in~\cite{Kargin_2020} for meandric systems implies a central limit theorem for the number of cycles in the product of two non-crossing fixed-point-free involutions.

A classical result of Goncharov~\cite{goncharov2} gives that the number of cycles in a uniformly random permutation of $S_n$ is asymptotically Gaussian, and hence the same holds for the product of two uniformly random permutations. Lugo~\cite{lugo2009} also proves asymptotic normality in the case of fixed-point-free involutions.

In light of these results, we have the following natural generalization.
\begin{question}
    Are there other natural families of subsets $\Gamma_n \subseteq S_n$ such that the number of cycles in the product of two uniformly random elements of $\Gamma_n$ is asymptotically Gaussian as $n \to \infty$?
\end{question}

\subsection{Enumerative Questions}
We ask the following question generalizing~\cref{thm:rational}.

\begin{question}
    For which infinite sets $J$ is $A ^ J(x, y)$ a rational function? In particular, is the unrestricted generating function $A ^ \NN(x, y)$ rational?
\end{question}

We conjecture that $A ^ J$ is not rational for any infinite $J$. For one heuristic, note that the arguments used to show $A ^ J$ is rational for finite $J$ do not generalize directly, since it is no longer true that only finite memory is required.

Likewise, some natural methods of reducing the size of seaweed meanders such as the \textit{winding-down} procedure given by Coll, Magnant, and Wang~\cite{coll2012signaturemeander} do not give rationality, since it is not enough to only track the first terms of $\lambda$ and $\mu$ in order to derive a recurrence.

We also have the following conjecture regarding the enumeration of seaweed meanders, which we have verified for $n \le 50$.
\begin{conj}
    Let $a_{n, k} = a_{n, k} ^ \NN$ be the number of seaweed meanders of size $n$ and index $k$. Then for any fixed $n$, the sequence $(a_{n, k})_{k = 1} ^ n$ is log-concave, that is,
    \[a_{n, k} ^ 2 \ge a_{n, k - 1} a_{n, k + 1}\]
    for all $1 < k < n$.
\end{conj}

Finally, it may be interesting to extend the results of~\cref{thm:triangle-billiard}. By combining~\cref{thm:rect-traj,thm:triangle-billiard} we see that for any seaweed meander $M = \frac{2\lambda}{2\mu}$, the billiard trajectories of $\rect(M)$ and $\tri(M)$ have a natural correspondence preserving the lengths of the trajectories.

\begin{question}
    Are there other families of chessboard polygons with corresponding triangular shapes such that the billiard trajectories in the square and triangular shapes have a natural correspondence?
\end{question}
\section{Acknowledgements}

This research was conducted at the University of Minnesota Duluth REU, which was funded by Jane Street Capital, Axiom Math AI, and a donation from Larry Penn. I am grateful to Joe Gallian and Colin Defant for organizing the REU and providing this opportunity.

I would like to thank Eliot Hodges, Claire Kaneshiro, Noah Kravitz, and Carl Schildkraut for their guidance and advice throughout the research process. I would also like to thank Mitchell Lee for helpful discussions, and Kerry Seekamp and Natalya Ter-Saakov for their feedback during the writing process.

\subsection{AI Usage Statement}

The major mathematical steps in this paper were done by the author alone, without AI assistance. 
LLM tools were used for literature search and programming, primarily the estimation of the constants $a$ and $b$ in~\cref{thm:clt-seaweed}. LLMs were also used for proofreading and editing of an initial draft of this paper. In particular, a minor error in the original proof of~\cref{thm:equidistrib} was identified at this stage.
\appendix
\crefalias{section}{appendix}
\section{Explicit Computations of Generating Functions}\label{appendix:bashing}

We answer a question asked by Archer, Geary, and Laudone~\cite{archer2025} by computing the generating function $A ^ J(x, y)$ for the sets $J = \{-1, 1, 2\}$, $J = \{0, 1, 3\}$, and  $J = \{1, 6\}$.
In the first two cases, we relate the generating functions to $A ^ {\{1, 2\}}$ and $A ^ {\{1, 3\}}$, respectively, while in the final case we compute a recursion using methods similar to those in~\cref{sec:genfunc}.

\begin{definition}
    For $J \subseteq \ZZ$, let $\calC_n ^ J$ denote the set of compositions of $n$ with all parts in $J_n$.
\end{definition}

\subsection{The case $J = \{-1, 1, 2\}$}

We first compute the generating function $A ^ {\{-1, 1, 2\}}$.

\begin{proof}[Proof of~\cref{thm:-112-gen-func}]
    We divide the seaweed meanders enumerated by $A ^ {\{-1, 1, 2\}}$ into the following disjoint cases, and compute the (multivariate) generating function for each:
    \begin{enumerate}[label=(\roman*)]
        \item $\frac{\lambda}{\mu}$, where $\lambda, \mu \in \calC_n ^ {\{1, 2\}}$;
        \item $\frac{1 | \lambda}{1 | n - 1}$, $\frac{\lambda | 1}{n - 1 | 1}$, $\frac{1 | n - 1}{1 | \lambda}$, or $\frac{n - 1 | 1}{\lambda | 1}$, where $\lambda \in \calC_{n - 1} ^ {\{1, 2\}}$ and $n \ge 4$;
        \item $\frac{2 | \lambda}{1 | n - 1}$, $\frac{\lambda | 2}{n - 1 | 1}$, $\frac{1 | n - 1}{2 | \lambda}$, or $\frac{n - 1 | 1}{\lambda | 2}$, where $\lambda \in \calC_{n - 2} ^ {\{1, 2\}}$ and $n \ge 4$;
        \item $\frac{1 | n - 1}{1 | n - 1}$ or $\frac{n - 1 | 1}{n - 1 | 1}$, where $n \ge 4$;
        \item $\frac{1 | n - 1}{n - 1 | 1}$ or $\frac{n - 1 | 1}{1 | n - 1}$, where $n \ge 4$.
    \end{enumerate}
    Note that each case consists of one family of seaweed meanders, up to reflections.

    \textbf{Case (i):} These seaweed meanders are exactly those where the parts of $\lambda$ and $\mu$ are restricted to $\{1, 2\}$, so the contribution is $A ^ {\{1, 2\}}(x, y)$.

    \textbf{Case (ii):} By symmetry, each of these four subcases has the same contribution to the total. Without loss of generality, we restrict our attention to the first case, of $\frac{1 | \lambda}{1 | n - 1}$, where $\lambda \in \calC_{n - 1} ^ {\{1, 2\}}$.
    Note that we have
    \[\frac{1 | \lambda}{1 | n - 1} = \frac{1}{1} \oplus \frac{\lambda}{n - 1},\] so
    \[\ind\left(\frac{1 | \lambda}{1 | n - 1}\right) = \ind\left(\frac{1}{1}\right) + \ind\left(\frac{\lambda}{n - 1}\right) = 1 + \ind\left(\frac{\lambda}{n - 1}\right).\]

    Let $r_{n, k}$ denote the number of seaweed meanders of the form $\frac{\lambda}{n}$ with index $k$, where $\lambda \in \calC_n ^ {\{1, 2\}}$, and let
    \[R(x, y) = \sum_{n, k \ge 0} r_{n, k} x ^ n y ^ k.\]
    The function $R(x, y)$ is computed in the proof of~\cite[Theorem 5.2]{archer2025}.

    The number of seaweed meanders of the form $\frac{1 | \lambda}{1 | n - 1}$ with index $k$ is then $r_{n - 1, k - 1}$. The leading terms of $R(x, y)$ (with degree in $x$ at most $2$) are $1 + xy + x ^ 2 y + x ^ 2y ^ 2$. We restrict to the case $n \ge 4$, so we must remove terms with degree in $x$ at most $2$. Thus, the generating function for all seaweed meanders enumerated by this case is
    \[4xy \left(R(x, y) - 1 - xy - x ^ 2y - x ^ 2 y ^ 2\right).\]

    \textbf{Case (iii):} Again by symmetry, it suffices to consider the case $\frac{2 | \lambda}{1 | n - 1}$. Note that the first vertex is connected to only the second vertex, and thus is the endpoint of a path. Contracting the edge between the first and second vertices, we have
    \[\ind\left(\frac{2 | \lambda}{1 | n - 1}\right) = \ind\left(\frac{1 | \lambda}{n - 1}\right).\]

    Now, we further refine the seaweed meanders enumerated by $r_{n, k}$ into several classes. We define the following quantities.
    \begin{itemize}
        \item Let $s_{n, k}$ be the number of seaweed meanders of size $n$ and index $k$ of the form $\frac{1 | \lambda | 1}{n}$ with $\lambda \in \calC_{n - 2} ^ {\{1, 2\}}$.
        \item Let $t_{n, k}$ be the number of seaweed meanders of size $n$ and index $k$ of the form $\frac{1 | \lambda | 2}{n}$ with $\lambda \in \calC_{n - 3} ^ {\{1, 2\}}$.
        \item Let $u_{n, k}$ be the number of seaweed meanders of size $n$ and index $k$ of the form $\frac{2 | \lambda | 1}{n}$ with $\lambda \in \calC_{n - 3} ^ {\{1, 2\}}$.
        \item Let $v_{n, k}$ be the number of seaweed meanders of size $n$ and index $k$ of the form $\frac{2 | \lambda | 2}{n}$ with $\lambda \in \calC_{n - 4} ^ {\{1, 2\}}$.
    \end{itemize}
    We define the multivariate generating functions $S(x, y)$, $T(x, y)$, $U(x, y)$, and $V(x, y)$ in the same way as for $R(x, y)$.
    By symmetry, we have $t_{n, k} = u_{n, k}$.

    By the above equality, for $n \ge 4$, the number of seaweed meanders of the form $\frac{2 | \lambda}{1 | n - 1}$ with index $k$ is then
    \[s_{n - 1, k} + t_{n - 1, k}.\]
    The only seaweed meander which is counted by $S(x, y)$ or $T(x, y)$ but has size less than $3$ is $\frac{1 | 1}{2}$, which has index $1$.

    Thus, the generating function for seaweed meanders of the form $\frac{2 | \lambda}{1 | n - 1}$ with $\lambda \in \calC_{n - 2} ^ {\{1, 2\}}$ is
    \[x(S(x, y) + T(x, y) - x ^ 2y).\]

    We now compute this quantity.
    For $n = 1$ and $n = 2$, the seaweed meanders $\frac{1}{1}$ and $\frac{2}{2}$ (respectively) are counted by $r_{n, k}$ but not by any of the four subclasses.
    For $n \ge 3$, we have
    \[r_{n, k} = s_{n, k} + t_{n, k} + u_{n, k} + v_{n, k}.\] 
    Thus,
    \[R(x, y) = S(x, y) + T(x, y) + U(x, y) + V(x, y) + 1 + xy + x ^ 2y ^ 2.\]
    In the seaweed meander $\frac{1 | \lambda | 1}{n}$, note that there is a path which consists of the first and last vertices. Thus, we have
    \[\ind\left(\frac{1 | \lambda | 1}{n}\right) = \ind\left(\frac{\lambda}{n - 2}\right) + 1,\] so $s_{n, k} = r_{n - 2, k - 1}$.
    Likewise, in the seaweed meander $\frac{2 | \lambda | 2}{n}$, there is a cycle consisting of the first two vertices and the last two vertices, so we have
    \[\ind\left(\frac{2 | \lambda | 2}{n}\right) = \ind\left(\frac{\lambda}{n - 4}\right) + 2\] and $v_{n, k} = r_{n - 4, k - 2}$.
    Hence, we have $S(x, y) = x ^ 2yR(x, y)$ and $V(x, y) = x ^ 4y ^ 2R(x, y)$.
    Thus,
    \[R(x, y) = (x ^ 2y + x ^ 4y ^ 2)R(x, y) + 2T(x, y) + 1 + xy + x ^ 2y ^ 2,\] so
    \[T(x, y) = \half \left((1 - x ^ 2y - x ^ 4y ^ 2)R(x, y) - 1 - xy - x ^ 2y ^ 2\right).\]
    As shown above, the desired generating function for this case then is equal to
    \[4x(S(x, y) + T(x, y) - x ^ 2y) = 2x \left((1 + x ^ 2y - x ^ 4y ^ 2)R(x, y) - 1 - xy - x ^ 2y ^ 2\right) - 4x ^ 3y.\]

    \textbf{Case (iv):} Note that we have
    \[\ind\left(\frac{1 | n - 1}{1 | n - 1}\right) = \ind\left(\frac{n - 1 | 1}{n - 1 | 1}\right) = n.\]
    Thus, since we sum over only $n \ge 4$, the generating function for this case is
    \[\frac{2x ^ 4y ^ 4}{1 - xy}.\]

    \textbf{Case (v):} By symmetry, it suffices to consider the seaweed meander $\frac{1 | n - 1}{n - 1 | 1}$.
    We claim that for $n \ge 4$,
    \[\ind\left(\frac{1 | n - 1}{n - 1 | 1}\right) = \begin{cases}
        1 & n \text{ odd}, \\
        2 & n \text{ even}.
    \end{cases}\]
    We use a similar argument as in the proof of~\cref{prop:big-parts-periodic}.

    Take $n \ge 4$. In the seaweed meander $\frac{1 | n - 1}{n - 1 | 1}$, note that the first and last vertices have degree $1$, and are connected to the $(n - 1)$th and $2$nd vertices, respectively. Contracting the edges $(1, n - 1)$ and $(2, n)$, we obtain the seaweed meander $\frac{1 | n - 3}{n - 3 | 1}$. In particular, we have
    \[\ind\left(\frac{1 | n - 1}{n - 1 | 1}\right) = \ind\left(\frac{1 | n - 3}{n - 3 | 1}\right).\]
    Thus, for odd $n$, we have
    \[\ind\left(\frac{1 | n - 1}{n - 1 | 1}\right) = \ind\left(\frac{1 | 2}{2 | 1}\right) = 1,\] and for even $n$, we have
    \[\ind\left(\frac{1 | n - 1}{n - 1 | 1}\right) = \ind\left(\frac{1 | 1}{1 | 1}\right) = 2.\]
    Hence, the generating function for this case is
    \[2\sum_{m \ge 2} (x ^ {2m}y ^ 2 + x ^ {2m + 1}y) = \frac{2(x ^ 4y ^ 2 + x ^ 5y)}{1 - x ^ 2}.\]

    Putting all these together, we have
    \begin{align*}
        A ^ {\{-1, 1, 2\}}(x, y) &= A ^ {\{1, 2\}}(x, y) + 2x(1 + 2y + x ^ 2y - x ^ 4y ^ 2)R(x, y) \\
        &\phantom{=} - 2x(1 + 2y + xy + 2xy ^ 2 + 2x ^ 2y + 3x ^ 2y ^ 2 + 2x ^ 2y ^ 3) + \frac{2x ^ 4y ^ 4}{1 - xy} + \frac{2(x ^ 4y ^ 2 + x ^ 5y)}{1 - x ^ 2}.
    \end{align*}
    Substituting in the values of $A ^ {\{1, 2\}}$ and $R(x, y)$ using the computations from~\cite[Theorems 3.2 and 5.2]{archer2025}\footnote{The rational function $R(x, y)$ in~\cite[Theorem 5.2]{archer2025} is incorrect as written, although the recurrence is correct. The correct formula is
    \[R(x, y) = \frac{1 + xy - x ^ 2 + x ^ 2y ^ 2 + x ^ 3y - x ^ 4y ^ 2}{1 - x ^ 2 - x ^ 2y - x ^ 4y - x ^ 4y ^ 2 + x ^ 6y ^ 2}.\]}gives the claimed formula.
\end{proof}

\subsection{The case $J = \{0, 1, 3\}$}

Next, we compute the generating function $A ^ {\{0, 1, 3\}}$. The proof follows roughly the same structure as the previous computation, where we divide the enumerated seaweed meanders into several cases, and compute the generating function for each case separately.

We first define and compute auxiliary generating functions which we will use in the proof.
\begin{itemize}
    \item Let $c_{n, k}$ be the number of seaweed meanders of size $n$ and index $k$ of the form $\frac{\lambda}{1 | \mu}$ with $\lambda \in \calC_n ^ {\{1, 3\}}$ and $\mu \in \calC_{n - 1} ^ {\{1, 3\}}$.
    \item Let $d_{n, k}$ be the number of seaweed meanders of size $n$ and index $k$ of the form $\frac{\lambda}{1 | 1 | \mu}$ with $\lambda \in \calC_n ^ {\{1, 3\}}$ and $\mu \in \calC_{n - 2} ^ {\{1, 3\}}$.
\end{itemize}
We also define the multivariate generating functions $C(x, y)$ and $D(x, y)$ in the same way as $A ^ J(x, y)$.

\begin{lemma}\label{lem:cd-gen-func}
    We have
    \[C(x, y) = \frac{xy}{1 - x ^ 2y - x ^ 3y} A ^ {\{1, 3\}}(x, y)\]
    and 
    \[D(x, y) = \frac{x ^ 2y ^ 2(1 + x)}{1 - x ^ 2y - x ^ 3y} A ^ {\{1, 3\}}(x, y).\]
\end{lemma}

\begin{proof}
    We derive recurrences by doing casework on the first part of $\lambda$. Let $\lambda'$ be $\lambda$ with the first part removed. Note that $\lambda'$ has all parts equal to $1$ or $3$.

    First, consider a seaweed meander of the form $\frac{\lambda}{1 | \mu}$. If $\lambda_1 = 1$, then
    \[\ind\left(\frac{\lambda}{1 | \mu}\right) = \ind\left(\frac{1}{1} \oplus \frac{\lambda'}{\mu}\right) = 1 + \ind\left(\frac{\lambda'}{\mu}\right).\]
    If $\lambda_1 = 3$, then vertex $1$ has degree $1$, and its only neighbor is vertex $3$. Contracting the edge $(1, 3)$, we obtain the seaweed meander $\frac{1 | 1 | \lambda'}{\mu}$. Thus,
    \[\ind\left(\frac{\lambda}{1 | \mu}\right) = \ind\left(\frac{1 | 1 | \lambda'}{\mu}\right) = \ind\left(\frac{\mu}{1 | 1 | \lambda'}\right),\] where the last equality follows from reflection over the midline.
    Hence,
    \[c_{n, k} = a_{n - 1, k - 1} ^ {\{1, 3\}} + d_{n - 1, k}.\]

    Now, consider a seaweed meander of the form $\frac{\lambda}{1 | 1 | \mu}$. If $\lambda_1 = 1$, then
    \[\ind\left(\frac{\lambda}{1 | 1| \mu}\right) = \ind\left(\frac{1}{1} \oplus \frac{\lambda'}{1 | \mu}\right) = 1 + \ind\left(\frac{\lambda'}{1 | \mu}\right).\]
    If $\lambda_1 = 3$, by the same argument as above,
    \[\ind\left(\frac{\lambda}{1 | 1 | \mu}\right) = \ind\left(\frac{1 | \mu}{1 | 1 | \lambda'}\right) = 1 + \ind\left(\frac{\mu}{1 | \lambda'}\right).\]
    Hence,
    \[d_{n, k} = c_{n - 1, k - 1} + c_{n - 2, k - 1}.\]
    Translating to the generating functions, we have
    \begin{align*}
        C(x, y) &= xyA ^ {\{1, 3\}}(x, y) + xD(x, y), \\
        D(x, y) &= (xy + x ^ 2y) C(x, y).
    \end{align*}
    Solving this system gives the claimed result.
\end{proof}

\begin{definition}
    For a composition $\lambda = (\lambda_1, \dots, \lambda_m) \Vdash n$, let $\lambda ^ r = (\lambda_m, \dots, \lambda_1) \Vdash n$ be the composition obtained by reversing $\lambda$.
\end{definition}

\begin{proof}[Proof of~\cref{thm:013-gen-func}]
    We divide the seaweed meanders enumerated by $A ^ {\{0, 1, 3\}}$ into the following disjoint cases:
    \begin{enumerate}[label=(\roman*)]
        \item $\frac{\lambda}{\mu}$, where $\lambda, \mu \in \calC_n ^ {\{1, 3\}}$;
        \item $\frac{\alpha | \beta}{n}$ or $\frac{n}{\alpha | \beta}$, where $\alpha, \beta \in \calC_{\frac{n}{2}} ^ {\{1, 3\}}$ and $n \ge 2$ is even;
        \item $\frac{\alpha | 3 | \beta}{n}$, $\frac{\beta | 3 | \alpha}{n}$, $\frac{n}{\alpha | 3 | \beta}$, or $\frac{n}{\beta | 3 | \alpha}$, where $\alpha \in \calC_{\frac{n}{2} - 1} ^ {\{1, 3\}}$, $\beta \in \calC_{\frac{n}{2} - 2} ^ {\{1, 3\}}$, and $n \ge 4$ is even;
        \item $\frac{\alpha | 1 | \beta}{n}$ or $\frac{n}{\alpha | 1 | \beta}$, where $\alpha, \beta \in \calC_{\frac{n - 1}{2}} ^ {\{1, 3\}}$ and $n \ge 5$ is odd;
        \item $\frac{\alpha | 3 | \beta}{n}$ or $\frac{n}{\alpha | 3 | \beta}$, where $\alpha, \beta \in \calC_{\frac{n - 3}{2}} ^ {\{1, 3\}}$ and $n \ge 5$ is odd;
        \item $\frac{\alpha | 3 | \beta}{n}$, $\frac{\beta | 3 | \alpha}{n}$, $\frac{n}{\alpha | 3 | \beta}$, or $\frac{n}{\beta | 3 | \alpha}$, where $\alpha \in \calC_{\frac{n - 1}{2}} ^ {\{1, 3\}}$, $\beta \in \calC_{\frac{n - 5}{2}} ^ {\{1, 3\}}$, and $n \ge 5$ is odd;
        \item $\frac{n}{n}$, where $n \ne 0, 1, 3$.
    \end{enumerate}
    To motivate this partition into cases, note that the arc diagram corresponding to $n$ gives edges of the form $(v, n + 1 - v)$ for $1 \le v \le \floor{\frac{n}{2}}$.
    Di Francesco, Golinelli, and Guitter~\cite{francesco_meanders_1997} define the notion of \textit{semi-meanders} and describe a process which corresponds to contracting each of the edges $(v, n + 1 - v)$.
    In particular, when we contract all of these edges, we obtain a graph which is similar to a seaweed meander on $\ceil{\frac{n}{2}}$ vertices. The behavior of the contraction is determined by the parts containing the vertices around $\frac{n}{2}$, which is exactly how we define the various cases.

    We now compute the generating function for each case separately.

    \textbf{Case (i):} These seaweed meanders are exactly those where the parts of $\lambda$ and $\mu$ are restricted to $\{1, 3\}$, so the contribution is $A ^ {\{1, 3\}}(x, y)$.

    \textbf{Case (ii):} Write $n = 2m$, where $m \in \NN$. Without loss of generality, we consider the seaweed meanders of the form $\frac{\alpha | \beta}{n}$.
    
    As mentioned above, we have edges $(v, 2m + 1 - v)$ for each $1 \le v \le m$. Now, contract each of these edges by rotating the vertices from $m + 1$ to $2m$ around the center of the diagram. This rotation sends the arcs which correspond to $\beta$ to arcs which lie below the midline, and reverses the order of the parts of $\beta$.
    Thus, upon contracting these edges, we obtain the seaweed meander $\frac{\alpha}{\beta ^ r}$.
    An example is shown in~\cref{fig:case-ii-contraction}.
    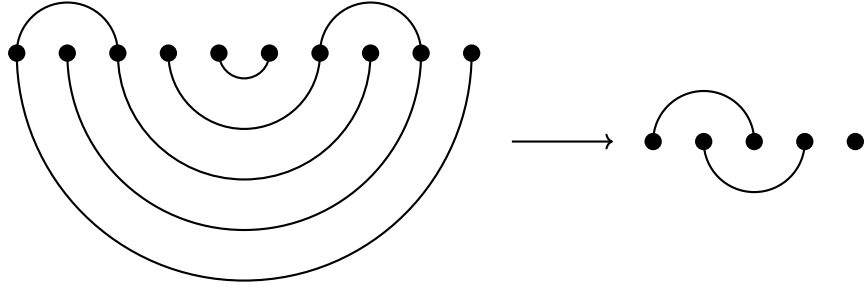
\begin{figure}[!ht]
        \begin{center}
            \resizebox{0.7\textwidth}{!}{%
                \begin{tikzpicture}[every node/.style={font=\ttfamily\Huge},every path/.style={draw, very thick},subfig/.style={}]
                    \node[subfig, anchor=east](A) at (0,0){
                        \begin{tikzpicture}
                            \foreach \x in {1,...,10}
                                \draw[fill] (\x,0) circle (0.15);
                            \foreach \x in {1,...,5}
                                \draw (\x,0) arc (-180:0:5.5-\x);
                            \draw (1, 0) arc (180:0:1);
                            \draw (7, 0) arc (180:0:1);
                        \end{tikzpicture}
                    };
                    \draw[->] (0.5, 0)--(2.5,0);
                    \node[subfig, anchor=west] (B) at (3,0){
                        \begin{tikzpicture}
                            \foreach \x in {1,...,5}
                                \draw[fill] (\x,0) circle (0.15);
                            \draw (1, 0) arc (180:0:1);
                            \draw (2, 0) arc (-180:0:1);
                        \end{tikzpicture}
                    };
                \end{tikzpicture}
        }
        \end{center}
        \caption{An example of the contraction in case (ii), with $n = 10$, $\alpha = (3, 1, 1)$, and $\beta = (1, 3, 1)$. The left shows the seaweed meander $\frac{\alpha | \beta}{n}$, and the right shows the seaweed meander $\frac{\alpha}{\beta ^ r}$.}
        \label{fig:case-ii-contraction}
    \end{figure}

    Thus, we have
    \[\ind\left(\frac{\alpha | \beta}{n}\right) = \ind\left(\frac{\alpha}{\beta ^ r}\right).\]
    Hence, the number of seaweed meanders of the form $\frac{\alpha | \beta}{n}$ with $\alpha, \beta \in \calC_m ^ {\{1, 3\}}$ and index $k$ is equal to $a_{m, k} ^ {\{1, 3\}}$.
    Thus, the generating function for this case is equal to
    \[2(A ^ {\{1, 3\}}(x ^ 2, y) - 1),\] where we remove the constant term since $m \ge 1$, and the factor of $2$ accounts for the symmetry in this case.
    
    \textbf{Case (iii):} As above, write $n = 2m$, with $m \ge 2$. By symmetry, each of the four patterns appears equally often, so without loss of generality, consider seaweed meanders of the form $\frac{\alpha | 3 | \beta}{n}$.

    Similar to the previous case, we contract the edges $(v, 2m + 1 - v)$ for each $1 \le v \le m$ by rotating the vertices $m + 1$ through $2m$ around the center. An example is shown in~\cref{fig:case-iii-contraction}.

    After performing the contraction, the $m$th vertex, which comes from identifying the $m$th and $(m + 1)$st vertices in the original seaweed meander, has degree $1$, and its only neighbor is vertex $m - 1$.
    Now, by contracting the edge $(m - 1, m)$, we obtain the seaweed meander $\frac{\alpha}{\beta ^ r | 1}$.
    \begin{figure}[!ht]
        \begin{center}
            \resizebox{0.9\textwidth}{!}{%
                \begin{tikzpicture}[every node/.style={font=\ttfamily\Huge},every path/.style={draw, very thick},subfig/.style={}]
                    \node[subfig, anchor=east](A) at (0,0){
                        \begin{tikzpicture}
                            \foreach \x in {1,...,10}
                                \draw[fill] (\x,0) circle (0.15);
                            \foreach \x in {1,...,5}
                                \draw (\x,0) arc (-180:0:5.5-\x);
                            \draw (1, 0) arc (180:0:1);
                            \draw (5, 0) arc (180:0:1);
                            \draw (8, 0) arc (180:0:1);
                        \end{tikzpicture}
                    };
                    \draw[->] (0.5, 0)--(2.5,0);
                    \node[subfig, anchor=west] (B) at (3,0){
                        \begin{tikzpicture}
                            \foreach \x in {1,...,5}
                                \draw[fill] (\x,0) circle (0.15);
                            \draw (1, 0) arc (180:0:1);
                            \draw (1, 0) arc (-180:0:1);
                            \draw (4,0)--(5,0);                            
                        \end{tikzpicture}
                    };
                    \draw[->] ($(B.east)+(0.5, 0)$)-- ++(2,0);
                    \node[subfig, anchor=west] (C) at ($(B.east) + (3, 0)$){
                        \begin{tikzpicture}
                            \foreach \x in {1,...,4}
                                \draw[fill] (\x,0) circle (0.15);
                            \draw (1, 0) arc (180:0:1);
                            \draw (1, 0) arc (-180:0:1);
                        \end{tikzpicture}
                    };
                \end{tikzpicture}
        }
        \end{center}
        \caption{An example of the contraction in case (iii), with $n = 10$, $\alpha = (3, 1)$, and $\beta = (3)$. The left shows the seaweed meander $\frac{\alpha | 3 | \beta}{n}$, the middle shows the seaweed meander after the first contraction step, and the right shows the seaweed meander $\frac{\alpha}{\beta ^ r | 1}$.}
        \label{fig:case-iii-contraction}
    \end{figure}
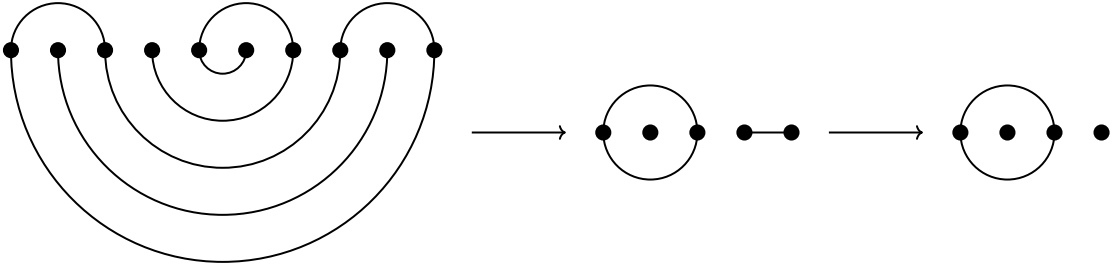

    Thus, we have
    \[\ind\left(\frac{\alpha | 3 | \beta}{n}\right) = \ind\left(\frac{\alpha}{\beta ^ r | 1}\right).\]
    Note that $\frac{\alpha}{\beta ^ r | 1}$ is a seaweed meander of size $m - 1$.

    By the above equality, the number of seaweed meanders of the form $\frac{\alpha | 3 | \beta}{n}$ with index $k$ is equal to $c_{m - 1, k}$. Thus, their generating function is equal to $x ^ 2C(x ^ 2, y)$.
    Hence, the generating function for this case is $4x ^ 2C(x ^ 2, y)$, where the factor of $4$ comes from the symmetry in this case.

    \textbf{Case (iv):} Consider a seaweed meander of the form $\frac{\alpha | 1 | \beta}{n}$.
    Note that the central vertex $\frac{n + 1}{2}$ has no adjacent edges. By removing this vertex, we get a graph which is isomorphic to the seaweed meander $\frac{\alpha | \beta}{n - 1}$. Hence,
    \[\ind\left(\frac{\alpha | 1 | \beta}{n}\right) = 1 + \ind\left(\frac{\alpha | \beta}{n - 1}\right).\]
    Using the result of case (ii), the generating function for this case is $2xy(A ^ {\{1, 3\}}(x ^ 2, y) - 1 - x ^ 2y)$. Note that we subtract $x ^ 2y$ since $n - 1 \ge 4$, while case (ii) counts seaweed meanders of size $2$.
    
    \textbf{Case (v):} The idea is similar to the previous case. The vertices $\frac{n - 1}{2}$ and $\frac{n + 3}{2}$ form a cycle of length $2$, and vertex $\frac{n + 1}{2}$ is isolated. By removing these vertices, we have
    \[\ind\left(\frac{\alpha | 3 | \beta}{n}\right) = 3 + \ind\left(\frac{\alpha | \beta}{n - 3}\right).\]
    Using the result of case (ii), the generating function for this case is $2x ^ 3y ^ 3(A ^ {\{1, 3\}}(x ^ 2, y) - 1)$.

    \textbf{Case (vi):} Write $n = 2m + 1$. Without loss of generality, consider a seaweed meander of the form $\frac{\alpha | 3 | \beta}{n}$, where $\abs{\alpha} = m$ and $\abs{\beta} = m - 2$. 

    We use a contraction argument similar to cases (ii) and (iii). An example is shown in~\cref{fig:case-iv-contraction}.

    Since $n$ is odd, the edges corresponding to the part of size $n$ in the seaweed meander are of the form $(v, 2m + 2 - v)$, for $1 \le v \le m$.
    We contract these edges by rotating vertices $m + 2$ to $2m + 1$ around the central vertex $m + 1$. The resulting graph has $m + 1$ vertices. Note that vertex $m + 1$ originally has degree $1$, with its only neighbor being vertex $m + 3$. After the edge contraction, its only neighbor is now vertex $m - 1$.

    Now, contract the edge $(m - 1, m + 1)$. Then the resulting graph is the seaweed meander $\frac{\alpha}{\beta ^ r | 1 | 1}$, since vertices $m - 1$ and $m$ do not have any edges remaining below the midline.
    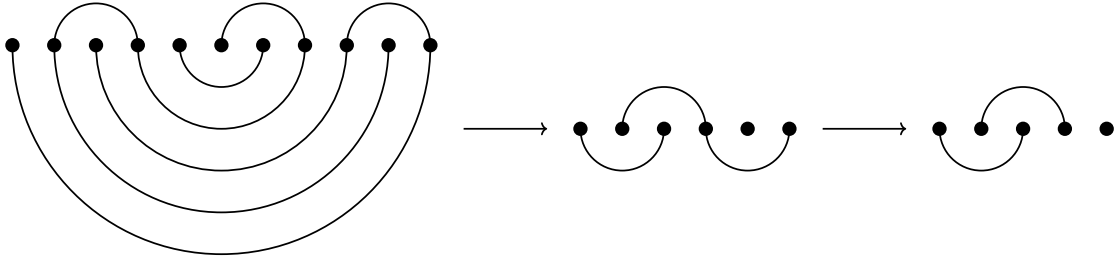
\begin{figure}[!ht]
        \begin{center}
            \resizebox{0.9\textwidth}{!}{%
                \begin{tikzpicture}[every node/.style={font=\ttfamily\Huge},every path/.style={draw, very thick},subfig/.style={}]
                    \node[subfig, anchor=east](A) at (0,0){
                        \begin{tikzpicture}
                            \foreach \x in {1,...,11}
                                \draw[fill] (\x,0) circle (0.15);
                            \foreach \x in {1,...,5}
                                \draw (\x,0) arc (-180:0:6-\x);
                            \draw (2, 0) arc (180:0:1);
                            \draw (6, 0) arc (180:0:1);
                            \draw (9, 0) arc (180:0:1);
                        \end{tikzpicture}
                    };
                    \draw[->] (0.5, 0)--(2.5,0);
                    \node[subfig, anchor=west] (B) at (3,0){
                        \begin{tikzpicture}
                            \foreach \x in {1,...,6}
                                \draw[fill] (\x,0) circle (0.15);
                            \draw (2, 0) arc (180:0:1);
                            \draw (1, 0) arc (-180:0:1);
                            \draw (4, 0) arc (-180:0:1);
                        \end{tikzpicture}
                    };
                    \draw[->] ($(B.east)+(0.5, 0)$)-- ++(2,0);
                    \node[subfig, anchor=west] (C) at ($(B.east) + (3, 0)$){
                        \begin{tikzpicture}
                            \foreach \x in {1,...,5}
                                \draw[fill] (\x,0) circle (0.15);
                            \draw (2, 0) arc (180:0:1);
                            \draw (1, 0) arc (-180:0:1);
                        \end{tikzpicture}
                    };
                \end{tikzpicture}
        }
        \end{center}
        \caption{An example of the contraction in case (vi), with $n = 11$, $\alpha = (1, 3, 1)$, and $\beta = (3)$. The left shows the seaweed meander $\frac{\alpha | 3 | \beta}{n}$, the middle shows the seaweed meander after the first contraction step, and the right shows the seaweed meander $\frac{\alpha}{\beta ^ r | 1 | 1}$.}
        \label{fig:case-iv-contraction}
    \end{figure}

    Note that $\frac{\alpha}{\beta ^ r | 1 | 1}$ is a seaweed meander of size $m$. Thus, the number of seaweed meanders of the form $\frac{\alpha | 3 | \beta}{n}$ with index $k$ is equal to $d_{m, k}$, so their generating function is $xD(x ^ 2, y)$.

    Hence, the contribution of this case is equal to $4xD(x ^ 2, y)$, where the factor of $4$ comes from the symmetry noted above.

    \textbf{Case (vii):} We have $\ind\left(\frac{n}{n}\right) = n$, so the generating function for this case is
    \[x ^ 2y ^ 2 + \frac{x ^ 4y ^ 4}{1 - xy}.\]

    Summing over all cases, we have
    \begin{align*}
        A ^ {\{0, 1, 3\}}(x, y) &= A ^ {\{1, 3\}}(x, y) + 2(A ^ {\{1, 3\}}(x ^ 2, y) - 1) + 4x ^ 2C(x ^ 2, y) + 2xy(A ^ {\{1, 3\}}(x ^ 2, y) - 1 - x ^ 2y) \\
        &\phantom{=} + 2x ^ 3y ^ 3(A ^ {\{1, 3\}}(x ^ 2, y) - 1) + 4xD(x ^ 2, y) + x ^ 2y ^ 2 + \frac{x ^ 4y ^ 4}{1 - xy}.
    \end{align*}
    Substituting the values of $C(x, y)$ and $D(x, y)$ from~\cref{lem:cd-gen-func} and $A ^ {\{1, 3\}}(x, y)$ from~\cite[Theorem 3.6]{archer2025} gives the desired result.
\end{proof}

\subsection{The case $J = \{1, 6\}$}

We now compute the generating function $A ^ {\{1, 6\}}$. We derive a recursion using a process similar to the one given in the proof of~\cref{prop:rational-positive}.

\begin{proof}[Proof of~\cref{thm:16-gen-func}]
    We first define the following sequences, which correspond to the states of the algorithm from the proof of~\cref{prop:rational-positive}. In each row of the table below, the sequence on the left counts the number of seaweed meanders given by the pattern on the right that have size $n$ and index $k$.
    \[
        \begin{array}{|c | c c|}
            \hline
            e_{n, k} & \frac{\lambda}{1 | \mu} & \lambda \in \calC_n ^ {\{1, 6\}}, \mu \in \calC_{n - 1} ^ {\{1, 6\}} \\
            \hline
            f_{n, k} & \frac{\lambda}{6 | \mu} & \lambda \in \calC_n ^ {\{1, 6\}}, \mu \in \calC_{n - 6} ^ {\{1, 6\}} \\
            \hline
            p_{n, k} & \frac{\lambda}{1 | 1 | \mu} & \lambda \in \calC_n ^ {\{1, 6\}}, \mu \in \calC_{n - 2} ^ {\{1, 6\}} \\
            \hline
            q_{n, k} & \frac{\lambda}{1 | 1 | 1 | \mu} & \lambda \in \calC_n ^ {\{1, 6\}}, \mu \in \calC_{n - 3} ^ {\{1, 6\}} \\
            \hline
            r_{n, k} & \frac{\lambda}{1 | 2 | \mu} & \lambda \in \calC_n ^ {\{1, 6\}}, \mu \in \calC_{n - 3} ^ {\{1, 6\}} \\
            \hline
            s_{n, k} & \frac{\lambda}{2 | \mu} & \lambda \in \calC_n ^ {\{1, 6\}}, \mu \in \calC_{n - 2} ^ {\{1, 6\}} \\
            \hline
            t_{n, k} & \frac{\lambda}{2 | 1 | \mu} & \lambda \in \calC_n ^ {\{1, 6\}}, \mu \in \calC_{n - 3} ^ {\{1, 6\}} \\
            \hline
            u_{n, k} & \frac{\lambda}{2 | 1 | 1 | \mu} & \lambda \in \calC_n ^ {\{1, 6\}}, \mu \in \calC_{n - 4} ^ {\{1, 6\}} \\
            \hline
            v_{n, k} & \frac{\lambda}{2 | 2 | \mu} & \lambda \in \calC_n ^ {\{1, 6\}}, \mu \in \calC_{n - 4} ^ {\{1, 6\}} \\
            \hline
            w_{n, k} & \frac{\lambda}{4 | 1 | \mu} & \lambda \in \calC_n ^ {\{1, 6\}}, \mu \in \calC_{n - 5} ^ {\{1, 6\}} \\
            \hline
        \end{array}
    \]
    Again, define the corresponding generating functions for each of these sequences. Throughout this proof, we write $a_{n, k} = a_{n, k} ^ {\{1, 6\}}$ and use $A(x, y)$ as shorthand for $A ^ {\{1, 6\}}(x, y)$.

    Note that for $n \ge 1$, we have $a_{n, k} = e_{n, k} + f_{n, k}$.

    We derive recurrences by doing casework on the first part of $\lambda$. As in the proof of~\cref{lem:cd-gen-func}, upon choosing the first part of $\lambda$, we can contract edges to obtain a seaweed meander of smaller size. Some of these contraction arguments may also be deduced using the winding-down procedure of Coll, Magnant, and Wang~\cite{coll2012signaturemeander}.
    
    As in the proof of~\cref{lem:cd-gen-func}, let $\lambda'$ be $\lambda$ with the first part removed. We derive a recursion for each row of the table above.

    \begin{itemize}
        \item Consider $\frac{\lambda}{1 | \mu}$. If $\lambda_1 = 1$, then
        \[\ind\left(\frac{\lambda}{1 | \mu}\right) = 1 + \ind\left(\frac{\lambda'}{\mu}\right).\]
        If $\lambda_1 = 6$, then vertex $1$ has degree $1$, and its only neighbor is vertex $6$. Contracting the edge $(1, 6)$, we obtain the seaweed meander $\frac{4 | 1 | \lambda'}{\mu}$, so
        \[\ind\left(\frac{\lambda}{1 | \mu}\right) = \ind\left(\frac{4 | 1 | \lambda'}{\mu}\right) = \ind\left(\frac{\mu}{4 | 1 | \lambda'}\right).\]
        These equalities imply
        \[e_{n, k} = a_{n - 1, k - 1} + w_{n - 1, k}.\]
        \item Consider $\frac{\lambda}{6 | \mu}$. If $\lambda_1 = 1$, by the same argument as in the previous case, we have
        \[\ind\left(\frac{\lambda}{6 | \mu}\right) = \ind\left(\frac{\lambda'}{4 | 1 | \mu}\right).\]
        If $\lambda_1 = 6$, we have
        \[\ind\left(\frac{\lambda}{6 | \mu}\right) = 6 + \ind\left(\frac{\lambda'}{\mu}\right).\]
        Thus,
        \[f_{n, k} = a_{n - 6, k - 6} + w_{n - 1, k}.\]
        \item Consider $\frac{\lambda}{1 | 1 | \mu}$. If $\lambda_1 = 1$, then
        \[\ind\left(\frac{\lambda}{1 | 1 | \mu}\right) = 1 + \ind\left(\frac{\lambda'}{1 | \mu}\right).\]
        If $\lambda_1 = 6$, then vertices $1$ and $2$ have degree $1$ and their neighbors are vertices $6$ and $5$, respectively. Contracting the edges $(1, 6)$ and $(2, 5)$, we obtain the seaweed meander $\frac{2 | 1 | 1 | \lambda'}{\mu}$, since vertices $3$ and $4$ are connected by an arc for $\lambda$.
        This contraction gives
        \[\ind\left(\frac{\lambda}{1 | 1 | \mu}\right) = \ind\left(\frac{2 | 1 | 1 | \lambda'}{\mu}\right) = \ind\left(\frac{\mu}{2 | 1 | 1 | \lambda'}\right).\]
        Hence,
        \[p_{n, k} = e_{n - 1, k - 1} + u_{n - 2, k}.\]
        \item Consider $\frac{\lambda}{1 | 1 | 1 | \mu}$.
        If $\lambda_1 = 1$, then
        \[\ind\left(\frac{\lambda}{1 | 1 | 1 | \mu}\right) = 1 + \ind\left(\frac{\lambda'}{1 | 1 | \mu}\right).\]
        If $\lambda_1 = 6$, then vertices $1$, $2$, and $3$ have degree $1$ and their neighbors are vertices $6$, $5$, and $4$ respectively. Contracting the edges $(1, 6)$, $(2, 5)$, and $(3, 4)$, we obtain the seaweed meander $\frac{1 | 1 | 1 | \lambda'}{\mu}$.
        This gives
        \[\ind\left(\frac{\lambda}{1 | 1 | 1 | \mu}\right) = \ind\left(\frac{1 | 1 | 1 | \lambda'}{\mu}\right) = \ind\left(\frac{\mu}{1 | 1 | 1 | \lambda'}\right).\]
        Thus,
        \[q_{n, k} = p_{n - 1, k - 1} + q_{n - 3, k}.\]
        \item Consider $\frac{\lambda}{1 | 2 | \mu}$. If $\lambda_1 = 1$, then
        \[\ind\left(\frac{\lambda}{1 | 2 | \mu}\right) = 1 + \ind\left(\frac{\lambda'}{2 | \mu}\right).\]
        If $\lambda_1 = 6$, then contracting the edges $(1, 6)$, $(2, 5)$, and $(3, 4)$ (by rotating vertices $1$, $2$ and $3$ around the midpoint of vertices $3$ and $4$) gives the seaweed meander $\frac{2 | 1 | \lambda'}{\mu}$, so
        \[\ind\left(\frac{\lambda}{1 | 2 | \mu}\right) = \ind\left(\frac{2 | 1 | \lambda'}{\mu}\right) = \ind\left(\frac{\mu}{2 | 1 | \lambda'}\right).\]
        The two cases together give
        \[r_{n, k} = s_{n - 1, k - 1} + t_{n - 3, k}.\]
        \item Consider $\frac{\lambda}{2 | \mu}$. If $\lambda_1 = 1$, vertex $1$ has degree $1$, and contracting the edge $(1, 2)$ gives the meander $\frac{\lambda'}{1 | \mu}$, so we have
        \[\ind\left(\frac{\lambda}{2 | \mu}\right) = \ind\left(\frac{\lambda'}{1 | \mu}\right).\]
        If $\lambda_1 = 6$, contracting edges $(1, 6)$ and $(2, 5)$ gives the seaweed meander $\frac{2 | 2 | \lambda'}{\mu}$, so
        \[\ind\left(\frac{\lambda}{2 | \mu}\right) = \ind\left(\frac{2 | 2 | \lambda'}{\mu}\right) = \ind\left(\frac{\mu}{2 | 2 | \lambda'}\right).\]
        Thus,
        \[s_{n, k} = e_{n - 1, k} + v_{n - 2, k}.\]
        \item Consider $\frac{\lambda}{2 | 1 | \mu}$. If $\lambda_1 = 1$, contracting the edge $(1, 2)$ gives the seaweed meander $\frac{\lambda'}{1 | 1 | \mu}$. Therefore,
        \[\ind\left(\frac{\lambda}{2 | 1 | \mu}\right) = \ind\left(\frac{\lambda'}{1 | 1 | \mu}\right).\]
        If $\lambda_1 = 6$, contracting the edges $(1, 6)$, $(2, 5)$, and $(3, 4)$ gives the seaweed meander $\frac{1 | 2 | \lambda'}{\mu}$, so
        \[\ind\left(\frac{\lambda}{2 | 1 | \mu}\right) = \ind\left(\frac{1 | 2 | \lambda'}{\mu}\right) = \ind\left(\frac{\mu}{1 | 2 | \lambda'}\right).\]
        Hence,
        \[t_{n, k} = p_{n - 1, k} + r_{n - 3, k}.\]
        \item Consider $\frac{\lambda}{2 | 1 | 1 | \mu}$. If $\lambda_1 = 1$, contracting the edge $(1, 2)$ gives the seaweed meander $\frac{\lambda'}{1 | 1 | 1 | \mu}$, so
        \[\ind\left(\frac{\lambda}{2 | 1 | 1 | \mu}\right) = \ind\left(\frac{\lambda'}{1 | 1 | 1 | \mu}\right).\]
        If $\lambda_1 = 6$, then vertices $3$ and $4$ form a connected component which is a path. Removing these vertices and contracting the edges $(1, 6)$ and $(2, 5)$ gives the seaweed meander $\frac{2 | \lambda'}{\mu}$, so we have
        \[\ind\left(\frac{\lambda}{2 | 1 | 1 | \mu}\right) = 1 + \ind\left(\frac{2 | \lambda'}{\mu}\right) = 1 + \ind\left(\frac{\mu}{2 | \lambda'}\right).\]
        Thus,
        \[u_{n, k} = q_{n - 1, k} + s_{n - 4, k - 1}.\]
        \item Consider $\frac{\lambda}{2 | 2 | \mu}$. If $\lambda_1 = 1$, contracting the edge $(1, 2)$ gives the seaweed meander $\frac{\lambda'}{1 | 2 | \mu}$, so
        \[\ind\left(\frac{\lambda}{2 | 2 | \mu}\right) = \ind\left(\frac{\lambda'}{1 | 2 | \mu}\right).\]
        If $\lambda_1 = 6$, then vertices $3$ and $4$ form a connected component which is a cycle. As in the previous case, removing these vertices and contracting the edges $(1, 6)$ and $(2, 5)$ gives the seaweed meander $\frac{2 | \lambda'}{\mu}$, so
        \[\ind\left(\frac{\lambda}{2 | 2 | \mu}\right) = 2 + \ind\left(\frac{2 | \lambda'}{\mu}\right) = 2 + \ind\left(\frac{\mu}{2 | \lambda'}\right).\]
        These equalities imply
        \[v_{n, k} = r_{n - 1, k} + s_{n - 4, k - 2}.\]
        \item Consider $\frac{\lambda}{4 | 1 | \mu}$. If $\lambda_1 = 1$, contracting the edge $(1, 4)$ gives the seaweed meander $\frac{\lambda'}{2 | 1 | 1 | \mu}$, so
        \[\ind\left(\frac{\lambda}{4 | 1 | \mu}\right) = \ind\left(\frac{\lambda'}{2 | 1 | 1 | \mu}\right).\]
        If $\lambda_1 = 6$, then we have a path consisting of the edges $(5, 2)$, $(2, 3)$, $(3, 4)$, $(4, 1)$, and $(1, 6)$ coming from the part of size $6$ in $\lambda$ and parts of size $4$ and $1$ in the bottom composition. Contracting all of these edges to a single vertex, we obtain the seaweed meander $\frac{1 | \lambda'}{\mu}$, so
        \[\ind\left(\frac{\lambda}{4 | 1 | \mu}\right) = \ind\left(\frac{1 | \lambda'}{\mu}\right) = \ind\left(\frac{\mu}{1 | \lambda'}\right).\]
        Hence,
        \[w_{n, k} = u_{n - 1, k} + e_{n - 5, k}.\]
    \end{itemize}
    Converting all of these recurrences to their corresponding generating functions, we obtain the following system of equations:
    \begin{align*}
        A(x, y) &= 1 + E(x, y) + F(x, y), \\
        E(x, y) &= xyA(x, y) + xW(x, y), \\
        F(x, y) &= x ^ 6y ^ 6A(x, y) + xW(x, y), \\
        P(x, y) &= xyE(x, y) + x ^ 2U(x, y), \\
        Q(x, y) &= xyP(x, y) + x ^ 3Q(x, y), \\
        R(x, y) &= xyS(x, y) + x ^ 3T(x, y), \\
        S(x, y) &= xE(x, y) + x ^ 2V(x, y), \\
        T(x, y) &= xP(x, y) + x ^ 3R(x, y), \\
        U(x, y) &= xQ(x, y) + x ^ 4yS(x, y), \\
        V(x, y) &= xR(x, y) + x ^ 4y ^ 2S(x, y), \\
        W(x, y) &= xU(x, y) + x ^ 5E(x, y).
    \end{align*}
    Solving the system for $A(x, y)$ gives the claimed result.
\end{proof}

\bibliographystyle{alpha}
\bibliography{sources}

\end{document}